\documentclass[a4paper,11pt]{article}
\usepackage{amsfonts}
\usepackage{amsmath,amssymb}
\usepackage{verbatim}
\usepackage{geometry}
\usepackage{url} 
\usepackage[colorlinks=true,pdfpagelabels,unicode]{hyperref}
\hypersetup{linkcolor=blue, urlcolor=blue, citecolor=blue}
\newtheorem{theo}{Theorem}[section]
\newtheorem{lem}[theo]{Lemma}
\newtheorem{cor}[theo]{Corollary}

\newtheorem{defi}[theo]{Definition}
\newcommand{\proof}{{\sc Proof.} \quad}

\newcommand{\be}{\begin{equation} \label}
\newcommand{\ee}{\end{equation}}
\newcommand{\bea}{\begin{eqnarray}\label}
\newcommand{\eea}{\end{eqnarray}}
\newcommand{\bas}{\begin{eqnarray*}}
\newcommand{\eas}{\end{eqnarray*}}
\newcommand{\bit}{\begin{itemize}}
\newcommand{\eit}{\end{itemize}}
\newcommand{\qed}{\hfill$\Box$ \vskip.2cm}
\newcommand{\nn}{\nonumber}
\newcommand{\R}{\mathbb{R}}
\newcommand{\N}{\mathbb{N}}
\newcommand{\pO}{\partial\Omega}

\newcommand{\eps}{\varepsilon}

\newcommand{\wto}{\rightharpoonup}
\newcommand{\wsto}{\stackrel{\star}{\rightharpoonup}}

\newcommand{\io}{\int_\Omega}
\newcommand{\na}{\nabla}
\newcommand{\Del}{\Delta}
\newcommand{\del}{\delta}
\newcommand{\lam}{\lambda}
\newcommand{\pa}{\partial}
\newcommand{\bom}{\overline{\Omega}}
\newcommand{\Om}{\Omega}

\newcommand{\hs}{\hspace*}

\newcommand{\vp}{\varphi}

\newcommand{\lbal}{\left\{ \begin{array}{l}}
\newcommand{\lball}{\left\{ \begin{array}{ll}}
\newcommand{\ear}{\end{array} \right.}

\newcommand{\abs}{\\[5pt]}

\newcommand{\tme}{T_{max,\eps}}

\newcommand{\ueps}{u_\eps}
\newcommand{\veps}{v_\eps}
\newcommand{\heps}{h_\eps}

\newcommand{\yeps}{y_\eps}
\newcommand{\zeps}{z_\eps}

\newcommand{\epsr}{\eps^\star}
\newcommand{\vepst}{v_{\eps t}}

\newcommand{\aeps}{a_{\eps}}
\newcommand{\beps}{b_{\eps}}
\newcommand{\meps}{m_{\eps}}
\newcommand{\fl}{f_\lambda}
\begin{document}
\enlargethispage{10mm}
\title{Subcritical-mass global solvability in a doubly degenerate Keller--Segel system
with signal production}
\author{
Tobias Black\footnote{tblack@math.uni-paderborn.de,}\\
{\small Institut f\"ur Mathematik, Universit\"at Paderborn,}\\
{\small 33098 Paderborn, Germany} 
\and
Genglin Li\footnote{genglin.li@hhu.edu.cn}\\
{\small School of Mathematics, Hohai University,}\\
{\small 211100 Nanjing, Jiangsu, P.R.~China}
}
\date{}
\maketitle
\begin{abstract}
\noindent 
We consider the initial-boundary value problem for a variant of the Keller--Segel chemotaxis system with doubly degenerate diffusion, i.e. we study
\begin{align*}
	\left\lbrace\begin{array}{l}
	u_t = \nabla\cdot (uv\na u) - \nabla\cdot (u^2 v\nabla v),\\
	v_t = \Delta v  +u -v,\\
	(uv\nabla u - u^2 v \nabla v)\cdot\nu=\nabla v\cdot \nu=0,\\
	u(x,0)=u_0(x), \quad v(x,0)=v_0(x),
	\end{array} \right.
\end{align*}
in a smoothly bounded domain $\Omega\subset\mathbb{R}^2$. Crucially, we only assume the sufficiently regular initial data to be nonnegative, but allow those functions to be zero at non-trivial parts of the domain. We show that, despite possibly starting from a degenerate state, the system admits global solutions in a framework of generalized energy solutions, whenever the initial mass is below $m_0=4\pi$. Moreover, in a radial setting the mass bound can be increased to $m_0=8\pi$. The solution components $u$ and $v$ attain their prescribed initial data strongly in $L^1(\Omega)$ and $L^p(\Omega)$, respectively, for every $p\in(1,\infty)$.

\end{abstract}
\newpage
%
%
%
%
\section{Introduction}\label{intro}
To capture the visible patterns emerging in physical experiments with populations of simple organisms has always been a driving motivation for the mathematical modeling. Take for example the observed morphological change in bacterial colonies witnessed during the evolution of \emph{B. subtilis} on agar plates with variations of environmental conditions (\cite{OhgiwaraMatsushitaMatsuyama92}), which motivated the formulation of the system featured in \cite{kawasaki_JTB1997}. In the proposed model a key assumption was that the motility of the bacteria is also reduced in areas with an insufficient nutrient supply. The authors of \cite{plaza} later augmented the model by an additional chemotaxis term, which represents a movement bias along concentration gradients towards areas with higher amount of nutrients. In this case the governing partial differential equations then essentially take the form

\be{consumpsys}
	\lball
	u_t = \nabla\cdot (uv\nabla u) - \nabla\cdot (u^2 v\nabla v)+f(u,v),\\
	v_t = \Delta v -uv,
	\end{array} \right.
\ee
where $u$ denotes the distribution of bacterial cells and $v$ the concentration of nutrients, respectively, and $f$ is a proliferation term.\smallskip
The analysis of the one-dimensional version of this system with prototypical choice $f(u,v)=uv$ in \cite{win_TRAN2021} revealed that solutions to this system archetype may actually feature large varieties of nontrivial steady states, a property well-sought in pursuit of recapturing pattern formation and mostly unprecedented in chemotaxis models without the degenerate synergy in the cell diffusion coefficient (\cite{taowin_subcrit,win_arma14,Lankeit16,winJDE17,winTRAN2017}). Subsequent work, mostly still concerned with this consumption-type framework
and closely related variants, has further underlined both the mathematical
subtlety and the qualitative relevance of the cross-degenerate structure.  In
one-dimensional settings, beyond the construction in \cite{win_TRAN2021}, the
same type of global bounded continuous solutions was obtained in
\cite{li_win_CPAA2021}, but for a larger class of initial data. For planar problems with strictly positive initial nutrient concentrations, the first global weak solutions in convex domains were
constructed under a suitable smallness condition in \cite{ct_doubly_degenerate_smallsignal}, while global bounded continuous weak
solutions were later obtained without that smallness restriction, but under the condition
$\int_\Omega\ln u_0>-\infty$ in \cite{zhang_li_m3as2026}, which was subsequently
removed in \cite{win_wu_preprint} by means of substantially more refined
techniques.  In \cite{win_wu_preprint}, convergence toward inhomogeneous limiting
profiles was also established, in line with the nontrivial large-time behavior
already observed in one dimension. Further advances concern modified systems in
 which the taxis
contribution is weakened or compensated by stronger dissipative effects
(\cite{black_kohatsu_wu_jee2026,deji_huang_wang_ejam2026,li_jde2022,win_jde2024,wu_nodea2024}),
as well as logistic extensions, where zero-order damping can instead drive
solutions toward spatially homogeneous equilibria
(\cite{li_win_aa2025,li_win_prsea2026,pan_narwa2024,pan_mu_na2026,zhang_li_amo2026}).
\smallskip
Accordingly, one might assume that, while mathematically difficult to treat, the degeneracy plays a major structural role in the occurrence of spatially non-homogeneous steady states. In nutrient models one can make the physically meaningful and mathematically convenient assumption that at the starting time at least some amount of nutrient is present everywhere in the domain in order to circumvent an initial degeneracy. The situation is quite different if we take a look at the other big class of chemotaxis systems, the Keller--Segel systems with signal production (\cite{kellerInitiationSlimeMold1970}), which were introduced to describe aggregation processes in slime mold populations. In these models the attracting chemical signal is only produced when the cells begin to starve and hence initially the signal concentration should not be assumed to be positive throughout the domain. 
\smallskip
Taking the above into account, we are going to concern ourselves with the system
\be{0}
	\lball
	u_t = \na\cdot (uv\na u) - \na\cdot (u^2 v\na v),
	\qquad & x\in\Om, \ t>0, \\[1mm]
	v_t = \Del v  +u -v,
	\qquad & x\in\Om, \ t>0, \\[1mm]
	(uv\na u - u^2 v \na v)\cdot\nu=\na v\cdot \nu=0,
	\qquad & x\in\pO, \ t>0, \\[1mm]
	u(x,0)=u_0(x), \quad v(x,0)=v_0(x),
	\qquad & x\in\Om,
	\end{array} \right.
\ee
in a bounded domain $\Om\subset\R^2$ with smooth boundary.\smallskip
As previously stated in the motivation above, one of our main interests is the possibility to start with an already degenerate state for the diffusion coefficient. Accordingly, we prescribe sufficiently regular initial data which are only assumed to be nonnegative. That is, we take initial data satisfying
  \be{init}	
	\lbal
	u_0\in W^{1,\infty}(\Om) \mbox{ is nonnegative with $u_0\not\equiv0$, and} \\[1mm]
	v_0\in W^{1,\infty}(\Om) \mbox{ is nonnegative in $\Om$.}
	\ear
  \ee
The signal-production setting considered here presents two related analytical challenges. First, the ratio of the taxis and diffusion coefficients in (\ref{0}), in their dependence on $u$, is $\frac{u^2v}{uv}=u$. This corresponds formally to the borderline growth $u^{\frac{2}{n}}$ in the theory of quasilinear Keller--Segel systems in two dimensions \cite{taowin_subcrit}, suggesting a delicate balance between chemotactic aggregation and the dissipation provided by cell diffusion.\smallskip
Second, a crucial technique in the analysis of the consumption system (\ref{consumpsys}) with strictly positive initial signal concentration is to exploit the evolution of
\bas
\int_\Omega v^{-\beta}|\nabla v|^q,
\qquad \beta,q>0,
\eas
to derive suitable estimates (see, for instance, \cite{ct_doubly_degenerate_smallsignal,win_jde2024}). However, this approach may no longer be available when signal consumption is replaced by signal production.\smallskip
These difficulties are compounded by the possibility that $v_0$ vanishes. Since the diffusion coefficient $uv$ depends on the signal concentration, such initial data may lead to degeneracy of cell diffusion at $t=0$, while the positive lower bound needed for higher-order estimates can only be established away from the initial time. We therefore seek global solvability in a generalized energy framework that accommodates the initial degeneracy and allows stronger regularity to be established for positive times. The following theorem establishes the existence of such solutions under suitable restrictions on the initial cell mass.
\begin{theo}\label{theo11}
  Let $\Om\subset\R^2$ be a smoothly bounded domain, and suppose that $u_0$ and $v_0$ satisfy (\ref{init}).
Then, if 
  \begin{itemize}
  \item[(i)] $\io u_0<4\pi$,\quad or
  \item[(ii)] $\Om=B_R(0)$ for some $R>0$ and $u_0$ and $v_0$ are radially symmetric with $\io u_0<8\pi$,
  \end{itemize}
	there exists at least one pair of functions
  \be{t1}
    \left\{
  \begin{array}{c}
  u\in C^0(\bom\times(0,\infty))
  \cap L^2_{loc}((0,\infty);W^{1,2}(\Om))\\
  v\in C^{2,1}(\bom\times(0,\infty))
  \cap L^\infty((0,\infty);W^{1,2}(\Om)), 
  \end{array}
    \right.
  \ee
 with $u\ge0$ on $\bom\times[0,\infty)$ and $v>0$ on $\bom\times(0,\infty)$,
  such that $(u,v)$ is a global generalized energy solution of 
   (\ref{0}) in the sense of Definition \ref{dw}. Moreover, 
  \be{t2}
  \sup_{t>0}\io u(\cdot,t)\ln u(\cdot,t)
  <\infty,
  \ee
  \be{t3}
  \io u(\cdot,t)
  =\io u_0
   \qquad\mbox{for all $t>0$,}
  \ee
  \be{t.u_initial}
  \lim_{t\searrow0}\|u(\cdot,t)-u_0\|_{L^1(\Om)}=0,
  \ee
  and, for every $p\in(1,\infty)$,
  \be{t4}
  \lim_{t\searrow 0}\|v(\cdot,t)-v_0\|_{L^p(\Om)}=0.
  \ee
For every $0<\tau<T$, the first equation in (\ref{0}) and its no-flux boundary condition hold on $\Om\times(\tau,T)$ in the weak sense specified by (\ref{14.u_weak}).
In case (ii) the solution is moreover radially symmetric about $x=0$ for all $t>0$.
\end{theo}
{\bf Remark.}\quad \textbf{(i)}
In view of (\ref{14.u_weak}), the first equation in (\ref{0}) holds in the ordinary
distributional sense on $\Om\times(\tau,T)$ for every $0<\tau<T$.
Since (\ref{14.u_weak}) holds for all test functions in
$C_0^\infty(\bom\times(\tau,T))$, it also expresses the no-flux boundary
condition in the weak sense.\smallskip

The second solution component $v$ solves
the second equation in (\ref{0}) classically for positive times. More precisely,
$v$ solves the boundary value problem
\bas
\left\{ 
\begin{array}{ll}
v_t = \Del v  +u -v,
\qquad & x\in\Om, \ t>0, \\[1mm]
\partial_\nu v=0,
\qquad & x\in\pO, \ t>0,
\end{array}\right.
\eas
classically in $\bom\times(0,\infty)$, together with the initial behavior
specified in (\ref{t4}). Indeed, this follows from the convergence of
$\veps$ toward $v$ established in (\ref{11.3}), the regularity assumption on
$v_0$ in (\ref{init}), and the initial behavior recorded in (\ref{11.9}). \smallskip

\textbf{(ii)}\quad
In the radially symmetric disk setting of Theorem~\ref{theo11}\textup{(ii)}, the
second author previously proved that the generalized energy
solution constructed in the present work satisfies
\[
u(\cdot,t)\rightharpoonup u_0
\quad\text{in }L^1(\Omega)
\quad\text{as }t\searrow0,
\]
including the case $v_0\equiv0$ \cite[Proposition~1.2]{li_jde2027}.
The present result \eqref{t.u_initial} strengthens this earlier conclusion
by establishing strong $L^1(\Omega)$ convergence. Moreover,
it extends the initial-trace result beyond the radially
symmetric setting to general smoothly bounded planar domains
under the mass condition in Theorem~\ref{theo11}\textup{(i)}.\bigskip

The sufficient mass bounds, featured in the conditions (i) and (ii) of Theorem~\ref{theo11}, coincide precisely with the critical mass in the non-radial and radial two-dimensional Keller--Segel system, respectively (\cite{nagai_senba_yoshida_1997,herrero_velazquez_1997,horstmann_wang_2001}). This coincidence is particularly intriguing in
 view of the structural similarities between the two systems: they share the same
  Lyapunov functional, and a further observation shows that they admit the same
   smooth, strictly positive stationary solutions. Moreover, blow-up is known to occur in the classical Keller--Segel system for supercritical mass (see \cite{herrero_velazquez_1997,horstmann_wang_2001}). These parallels naturally raise the question of whether a corresponding sharp critical-mass phenomenon occurs in the present setting. However, for \eqref{0} there is currently no corresponding blow-up result available, so the answer to this question is still open and under consideration.\abs
{\bf Main ideas.} \quad
The foundation for our argument is the observation that, despite the degeneracy present in the diffusion coefficient of the first equation in \eqref{0}, the Lyapunov-functional known from the standard Keller--Segel system remains intact (Lemma~\ref{LF}) and can be exploited to obtain a first crucial set of time global bounds (Lemma~\ref{lem1}). To derive a priori estimates of higher regularity from these bounds, however, we are forced to draw on a uniform lower bound for the chemical concentration, which due to our assumption that the initial data are only nonnegative can only be obtained away from the starting time (Lemma~\ref{lem2}). Accordingly, we have to restrict the derivation of higher order regularity estimates to time-intervals not including zero (Lemma~\ref{lem17} and Lemma~\ref{lem7}). \abs
To pass to the limit in the regularized systems, we also need to establish strong convergence of the gradients of the cell densities. In Lemma~\ref{lem10}, we achieve this by adapting the method developed in \cite{li_win_prsea2026} for the corresponding signal-consumption setting with $v_0>0$, to the current framework with delayed positivity and initial degeneracy.\abs
With these convergence properties at hand, Section~\ref{s6} establishes that the limiting pair constructed in Lemma~\ref{lem11} is a global generalized energy solution. More precisely, we combine the identity obtained through bounded renormalizations (Lemma~\ref{lem12}) and the associated integrability estimates (Lemma~\ref{lem13}) with the energy bounds from Lemma~\ref{lem1} to establish the renormalized formulation incorporating the prescribed initial data and to pass to the limit in the energy inequality (Lemma~\ref{lem14}). \abs
Finally, Section~\ref{s7} establishes strong attainment of the initial cell density in $L^1(\Om)$ by combining the renormalized inequality with convexity and mass conservation.
%
%

%
%
\setcounter{equation}{0}
\section{Local solutions to a family of approximate systems.}\label{s2}
The generalized solutions described in Theorem~\ref{theo11} will be constructed as a limit of functions solving a regularized variant of \eqref{0}. Since we allow the initial data to be zero throughout the domain, we intend to avoid the degeneracy at the starting time and therefore add a small positive perturbation to both initial distributions. Accordingly, we are going to consider the following family of systems:

\be{0eps}
	\left\{ \begin{array}{ll}
	u_{\eps t} = \na\cdot (\ueps\veps\na\ueps) - \na\cdot (\ueps^2 \veps \na\veps) ,
	\qquad & x\in\Om, \ t>0, \\[1mm]
	v_{\eps t} = \Del \veps + \ueps -\veps,
	\qquad & x\in\Om, \ t>0, \\[1mm]
	\frac{\pa\ueps}{\pa\nu}=\frac{\pa\veps}{\pa\nu}=0,
	\qquad & x\in\pO, \ t>0, \\[1mm]
	\ueps(x,0)=u_{0\eps}:=u_0(x)+\eps, \quad \veps(x,0)=v_{0\eps}:=v_0(x)+\eps,
	\qquad & x\in\Om.
	\end{array} \right.
\ee
For these regularized systems the time-local well-posedness and an extensibility criterion for the solutions can readily be obtained by standard arguments featured in closely related settings of doubly-degenerate chemotaxis problems.
\begin{lem}\label{lem_loc}
  Assume $\Om\subset\R^2$ is a domain with smooth boundary and that $(u_0,v_0)$ comply with \eqref{init}.
  Then for each $\eps\in (0,1)$, there exist $\tme\in (0,\infty]$ as well as functions
  \be{l1}
	\left\{ \begin{array}{l}
	\ueps \in C^0(\bom\times [0,\tme)) \cap C^{2,1}(\bom\times (0,\tme))	\qquad \mbox{and} \\[1mm]
	\veps \in C^0(\bom\times [0,\tme)) \cap C^{2,1}(\bom\times (0,\tme))
	\end{array} \right.
  \ee
  such that $\ueps>0$ and $\veps>0$ in $\bom\times [0,\tme)$,
  that $(\ueps,\veps)$ uniquely solves (\ref{0eps}) in the classical sense, and that
  \be{ext}
	\mbox{if $\tme<\infty$, \quad then \quad}
	\limsup_{t\nearrow\tme} \|\ueps(\cdot,t)\|_{L^\infty(\Om)} 
	= \infty.
  \ee
  Furthermore, if $\Om=B_R(0)$ for some $R>0$ and if $u_0$ and $v_0$ are
   radially symmetric, then the corresponding solution
    $(\ueps(\cdot,t),\veps(\cdot,t))$ is also
   radially symmetric about $x=0$ for all $t\in(0,\tme)$.
\end{lem}
\proof
   The local classical solvability follows from standard parabolic theory (cf.~\cite{amann, amann1993}).
   The extensibility criterion (\ref{ext}) can be derived by adapting the argument
   in \cite{ct_doubly_degenerate_smallsignal}, combined with the availability of a
   strictly positive lower bound for $\veps$. The latter follows from the definition
   of $v_{0\eps}$ together with \cite[Lemma~3.1]{fujie_senba_nonlinearity2018}. 
   Finally, if $\Omega=B_R(0)$ and $u_0$ and $v_0$ are radially symmetric about $x=0$,
   then the radial symmetry of $(\ueps,\veps)$ for all $t\in(0,\tme)$
   is a direct consequence of the uniqueness of the classical solutions.
\qed
\setcounter{equation}{0}
\section{A priori estimates on the full interval $(0,\tme)$}\label{s3}

Integration of the equation in \eqref{0} immediately provides uniform $L^1$-information as a starting point for further analysis. More importantly, the cell flux in \eqref{0} factors as
\bas
J:=uv\na u-u^2v\na v=u^2v\na(\ln u-v),
\eas
so the driving potential $\ln u-v$ is the same as in the classical
Keller--Segel system, while the mobility is $u^2v$ rather than $u$.
The classical functional $\mathcal{F}$, defined below, therefore
retains its energy-dissipation identity, with cell dissipation
\bas
\io u^2v|\na(\ln u-v)|^2
=\io \Big|\sqrt{v}\na u-u\sqrt{v}\na v\Big|^2.
\eas
However, obtaining space--time $L^1$ control of $J$ from this
dissipation by H\"older's inequality would require an additional space-time
bound for $\io u^2v$, beyond the basic energy
estimates.This additional requirement poses an obstacle to a direct 
weak-solution construction and motivates the generalized energy formulation adopted below, which does not require full-flux control through the initial time.
\begin{lem}\label{LF}
Let $\Omega\subset \R^{2}$ and $m>0$. Assume $(u_0,v_0)$ comply with \eqref{init} and satisfy $\io u_0=m$, 
\be{mass}
\io \ueps(\cdot,t)=\meps:=m+\eps|\Om|
\qquad \mbox{for all $t\in (0,\tme)$ and } \eps\in (0,1),
\ee
and 
\be{vmass}
\io \veps(\cdot,t)
\le \max\Big\{\io v_0, m\Big\}+|\Om|
\qquad \mbox{for all $t\in (0,\tme)$ and } \eps\in (0,1).
\ee
Moreover, 
\be{lf1}
 \int_{0}^{t}\io v_{\eps t}^2
 + \int_{0}^{t}\io \Big|\sqrt{\veps}\na\ueps-\ueps\sqrt{\veps}\na\veps\Big|^2
 +\mathcal{F}(\ueps(\cdot,t),\veps(\cdot,t))=\mathcal{F}(u_{0\eps},v_{0\eps})
\ee
for all $t\in (0,\tme)$ and $\eps\in (0,1)$, where the functional $\mathcal{F}$ is given by
\be{df}
\mathcal{F}(\phi,\psi)
:=\io \phi\ln\phi -\io \phi\psi + \frac{1}{2} \io |\na \psi|^2
+\frac{1}{2} \io \psi^2,
\ee
for $(\phi,\psi)$ in the finite-energy class
\[
\mathcal{E}:=\left\{
(\phi,\psi)\in L^1(\Om)\times W^{1,2}(\Om):
\begin{array}{l}
\phi\ge0\ \mbox{a.e. in }\Om,\\
\phi\ln\phi\in L^1(\Om),\quad \phi\psi\in L^1(\Om)
\end{array}
\right\}.
\]
Here and throughout, we use the convention $\phi\ln\phi=0$ wherever $\phi=0$.
\end{lem}
\proof
Integrating the first equation in (\ref{0eps}) over $\Om$ immediately shows that
 $\frac{d}{dt}\io \ueps=0$, and thus (\ref{mass}) follows. Using this together with 
 the definitions of $u_{0\eps}$ and $v_{0\eps}$,
 an integration of the second equation in (\ref{0eps}) over $\Om$ and an ODE comparison yield (\ref{vmass}).\medskip
 
 Next, an integration by parts together with a straightforward computation applied to (\ref{0eps}) gives
  \bea{LF.1}
  \frac{d}{dt}\io \big(\ueps\ln\ueps - \ueps\veps\big) 
  =  -\io\veps|\na\ueps|^2 + 2\io \ueps\veps\na\ueps\cdot\na\veps
  - \io \ueps^2\veps|\na\veps|^2 - \io \ueps v_{\eps t}
  \eea
  for all $t\in (0,\tme)$. To handle the last integral on the right-hand side, we use the second equation in
   (\ref{0eps}), and then integrate by parts to obtain
  \bea{Lf.2}
   - \io \ueps v_{\eps t}
   = \io v_{\eps t} (-v_{\eps t} + \Del \veps -\veps)
   = -\io v_{\eps t}^2 - \frac{1}{2}\frac{d}{dt} \io |\na \veps|^2 -  \frac{1}{2}\frac{d}{dt} \io \veps^2
  \eea
   for all $t\in (0,\tme)$. Moreover, observing that
   \bas
   \io\veps|\na\ueps|^2 
   - 2\io \ueps\veps\na\ueps\cdot\na\veps
      + \io \ueps^2\veps|\na\veps|^2
      = \io \Big|\sqrt{\veps}\na\ueps-\ueps\sqrt{\veps}\na\veps\Big|^2,
   \eas
 we substitute \eqref{Lf.2} into \eqref{LF.1}. After collecting the resulting terms and
 integrating over $(0,t)$, we arrive at \eqref{lf1}.
\qed
To shorten some the notations we introduce the following constant
\begin{align*}
m_0:=\begin{cases}
8\pi,\qquad & \mbox{if }\Om=B_R(0)\mbox{ for some }R>0\mbox{ and }(u_0,v_0)\mbox{ is radial about } x=0, \\
4\pi,\qquad & \mbox{otherwise,}
\end{cases}
\end{align*}
to easily distinguish between the initial mass condition in the cases (i) and (ii) of Theorem~\ref{theo11}. Throughout the remainder we assume that for an arbitrary mass-level $m<m_0$, we have fixed initial data $(u_0,v_0)$ complying with \eqref{init} and satisfying the standing hypothesis $$m=\io u_0<m_0.$$ The pair of functions $(\ueps,\veps)$ will denote the corresponding solution to \eqref{0eps} as obtained in Lemma~\ref{lem_loc} and $\tme$ the maximal existence time established therein. The number $\meps$ will always represent the constant provided by Lemma~\ref{LF}.
Unless otherwise stated, constants may depend on
these fixed initial data and on the fixed domain; this dependence
will be suppressed in the notation.
\smallskip 

Combining the functional from the previous Lemma with the famous Moser--Trudinger inequality, we can now derive uniform bounds whenever $\eps$ is sufficiently small and the initial mass is below $m_0$.
\begin{lem}\label{lem1}
There exist $C>0$ and $\epsr\in(0,1)$ such that the estimates
\be{1.1}
\io \ueps(\cdot,t)\ln\ueps(\cdot,t)\le C
\ee
and
\be{1.2}
\|\veps(\cdot,t)\|_{W^{1,2}(\Om)}\le C
\ee
as well as
\be{1.3}
\io \ueps(\cdot,t)\veps(\cdot,t)\le C
\ee
and
\be{1.4}
\int_0^t\io \vepst^2(\cdot,s)ds \le C
\ee
and
\be{1.6}
\int_0^t\io \Big|\sqrt{\veps}(\cdot,s)\na\ueps(\cdot,s)-\ueps(\cdot,s)\sqrt{\veps(\cdot,s)}\na\veps(\cdot,s)\Big|^2ds \le C
\ee
hold for all $t\in(0,\tme)$ and $ \eps \in (0, \epsr)$.
\end{lem}
\proof
  Since $m<m_0$, we can first choose $\del=\del(m)>0$ sufficiently small so that the inequality
  \bas
  \frac{1}{2}>m\cdot\frac{1}{2m_0}\cdot(1+\del)^2
  \eas
  remains valid, whereafter we may also fix $\epsr=\epsr(m)\in(0,1)$ satisfying
   \bas
    \frac{1}{2}>(m+\epsr|\Om|)\cdot\frac{1}{2m_0}\cdot(1+\del)^2.
    \eas
 We then select $\eta=\eta(m)>0$ such that
   \bas
      \frac{1}{2}>(m+\epsr|\Om|)\cdot\Big(\frac{1}{2m_0}+\eta\Big)\cdot(1+\del)^2,
   \eas
  which ensures that for all $\eps\in(0,\epsr)$,
  \be{1.65}
  \frac{1}{2}>(m+\eps|\Om|)\cdot\Big(\frac{1}{2m_0}+\eta\Big)\cdot(1+\del)^2.
  \ee
  Next we follow the arguments of \cite[Lemma 3.4]{nagai_senba_yoshida_1997} and \cite[Theorem 3.1]{bbtw_ks_review}. Again denoting $\meps=m+\eps|\Omega|$ we find from Jensen's inequality and \eqref{mass}, that 
  \bas
 -\ln\Big\{\frac{1}{\meps}\io e^{(1+\del)\veps}\Big\}
 &\le& -\io \ln\Big(\frac{e^{(1+\del)\veps}}{\ueps}\Big)\cdot \frac{\ueps}{\meps}\\
 &=& -\frac{1}{\meps}\Big\{(1+\del)\io \ueps\veps-\io \ueps\ln\ueps\Big\}
  \eas
 and hence
 \be{1.7}
-\io\ueps\veps
+\io \ueps\ln\ueps
 \ge \del \io\ueps\veps
 - \meps\ln\io e^{(1+\del)\veps} + \meps\ln \meps
 \ee
for all $t\in(0,\tme)$ and $\eps\in(0,\epsr)$. The Moser--Trudinger inequality (see e.g. \cite[Theorem 2.1]{nagai_senba_yoshida_1997}) provides $c_1=c_1(m)>0$ satisfying, 
 for all $t\in(0,\tme)$ and $\eps \in(0,\epsr)$,
 \bas
 \io e^{(1+\del)\veps}
 \le c_1\exp\bigg\{\Big(\frac{1}{2m_0}+\eta\Big)\cdot(1+\del)^2\|\na\veps\|_{L^2(\Om)}^2
 +\frac{2(1+\del)}{|\Om|}||\veps||_{L^1(\Om)}\bigg\}.
 \eas
 Using (\ref{vmass}), this implies 
 \bas
 \ln  \io e^{(1+\del)\veps}
 &\le& \Big(\frac{1}{2m_0}+\eta\Big)\cdot(1+\del)^2
 \cdot \io|\na\veps|^2
 +c_2
 \eas 
where $c_2=c_2(m):=\frac{2(1+\del)}{|\Om|}\cdot
 \big\{\max\{\io v_0, m\}+|\Om|\big\}
+\ln c_1$.  

\medskip

In view of (\ref{df}), (\ref{1.65}) and (\ref{1.7}), we infer that
\bas
\mathcal{F}(\ueps,\veps)
&\ge& \frac{1}{2}\io |\na\veps|^2
-\io\ueps\veps
+\io \ueps\ln \ueps\\
&\ge& \frac{1}{2}\io |\na\veps|^2
+\del \io\ueps\veps
 - \meps\ln\io e^{(1+\del)\veps} + \meps\ln \meps\\
 &\ge&\Big\{\frac{1}{2}-\meps\cdot\Big(\frac{1}{2m_0}+\eta\Big)\cdot(1+\del)^2\Big\} \io |\na\veps|^2
 + \del \io\ueps\veps
 -c_2\meps-\frac{1}{e}\\
 &\ge& \del \io\ueps\veps
  -c_3
  \qquad\mbox{for all $t\in(0,\tme)$ and $\eps\in(0,\epsr)$.}
\eas
where we used $c_3=c_3(m):=c_2(m_0+|\Om|)+\frac{1}{e}> c_2\meps+\frac{1}{e}$ for all $\eps\in(0,\epsr)$.

Consequently, (\ref{1.3}) follows from (\ref{lf1}) together with the above lower bound
for $\mathcal{F}(\ueps,\veps)$ and the uniform boundedness of
$\mathcal{F}(u_{0\eps},v_{0\eps})$ for all $\eps\in(0,\epsr)$,
which in turn is ensured by the regularity assumption in (\ref{init}).
Finally, upon combining (\ref{1.3}) with (\ref{lf1}),
we readily infer the remaining estimates
(\ref{1.1}), (\ref{1.2}), (\ref{1.4}), and (\ref{1.6}).
\qed
For the rest of the work $\epsr$ will always denote the constant obtained in Lemma~\ref{lem1}. If not explicitly stated otherwise the presented (in-)equalities will be valid for all $\eps\in(0,\epsr)$ and the provided constants will be independent of $\eps\in(0,\epsr)$. \smallskip

With the control over the gradient of $\veps$ provided by the previous lemma, we can now establish the following differential inequality for higher order terms.
\begin{lem}\label{lem5}
Let $p>1$. Then there exists $C(p)>0$ such that $t\in (0,\tme)$,
\bea{5.1}
 & &\hs{-10mm}
 \frac{d}{dt}\bigg\{\io \ueps^p
    + \io |\na\veps|^{2p+2}\bigg\}
    +\frac{p(p-1)}{2}\io \ueps^{p-1}\veps |\na\ueps|^2
    +\frac{2p}{p+1}\io |\na|\na\veps|^{p+1}|^2\nn\\
& &\le \frac{p(p-1)}{2}\io \ueps^{p+1} \veps |\na\veps|^2
    + C(p)\io \ueps^2|\na\veps|^{2p} + C(p).
\eea
\end{lem}
\proof
     By (\ref{0eps}) and Young's inequality, we have
     \bea{3.2}
     \frac{1}{p}\frac{d}{dt}\io \ueps^p
     &=& -(p-1)\io \ueps^{p-2}\na\ueps
     \cdot\{\ueps\veps\na\ueps-\ueps^2\veps\na\veps\}\nn\\
     &\le& -\frac{p-1}{2}\io \ueps^{p-1}\veps |\na\ueps|^2
     + \frac{p-1}{2}\io \ueps^{p+1} \veps |\na\veps|^2.
     \eea
     for all $t\in (0,\tme)$.
     Next, using the identities $\na\veps\cdot\na\Del\veps=\frac{1}{2}\Del|\na\veps|^2-|D^2\veps|^2$\newline
     and $\frac{p+1}{2}|\na\veps|^{p-1}\na|\na\veps|^2=\na|\na\veps|^{p+1}$, 
     and applying the second equation in (\ref{0eps}), we integrate by parts 
     multiple times to obtain
      \bea{3.3}
      \frac{1}{2p+2}\frac{d}{dt}\io |\na\veps|^{2p+2}
      &=& \io |\na\veps|^{2p}\na\veps\cdot\na\{\Del\veps+\ueps-\veps\}\nn\\
      &=& -\frac{2p}{(p+1)^2}\io |\na|\na\veps|^{p+1}|^2
      -\io |\na\veps|^{2p}|D^2\veps|^2\nn\\
      & &-2p\io \ueps|\na\veps|^{2p-2}(D^2\veps\cdot\na\veps)\cdot\na\veps
      -\io \ueps|\na\veps|^{2p}\Del\veps\nn\\
      & &-\io|\na\veps|^{2p+2}+\frac{1}{2}\int_{\pO}|\na\veps|^{2p}\frac{\pa|\na\veps|^2}{\pa\nu}
      \eea
       for all $t\in (0,\tme)$. Since $|\Del\veps|\le \sqrt{2}|D^2\veps|$, Young's inequality yields
      \bea{3.4}
      & &\hs{-25mm}
      -2p\io \ueps|\na\veps|^{2p-2}(D^2\veps\cdot\na\veps)\cdot\na\veps
         -\io \ueps|\na\veps|^{2p}\Del\veps\nn\\
       &\le& \left(2p+\sqrt{2}\right)\io \ueps|\na\veps|^{2p}|D^2\veps|\nn\\
       &\le& \io |\na\veps|^{2p}|D^2\veps|^2+ c_1(p)\io \ueps^2|\na\veps|^{2p}
      \eea
      for some $c_1(p)>0$. To deal with the boundary term in (\ref{3.3}), we use  
       $\frac{\pa\veps}{\pa\nu}=0$ on $\pO$ 
       and the fact that there exists $c_2>0$ such that $\frac{\pa|\na\veps|^2}{\pa\nu}\le c_2|\na \veps|^2$ on $\pO$
       for all $t\in (0,\tme)$ (\cite{lions_arma1980}).
      By the compact embedding $W^{r+\frac12,2}(\Om)\hookrightarrow L^2(\pO)$
      for any $r\in(0,\frac{1}{2})$, we get
      \be{3.5}
      \frac{1}{2}\int_{\pO}|\na\veps|^{2p}\frac{\pa|\na\veps|^2}{\pa\nu}
      \le \frac{p}{(p+1)^2}\io |\na|\na\veps|^{p+1}|^2
      + c_3(p) \Big\{\io|\na\veps|^2\Big\}^{p+1}
      \ee
      for some $c_3(p)>0$ and all $t\in (0,\tme)$.
      Combining (\ref{3.3})-(\ref{3.5}) and using (\ref{1.2}), we find $c_4(p)>0$ fulfilling 
      \be{3.6}
      \frac{d}{dt}\io |\na\veps|^{2(p+1)}
      +\frac{2p}{p+1}\io |\na|\na\veps|^{p+1}|^2
      \le c_4(p)\io \ueps^2|\na\veps|^{2p} + c_4(p)
      \ee
    for all  $t\in (0,\tme)$. Finally, adding (\ref{3.2}) to
     (\ref{3.6}) yields (\ref{5.1}).   
      \qed
\setcounter{equation}{0}
\section{Estimates away from the initial time}\label{s4}
Unfortunately, the remaining terms on the right-hand side of \eqref{5.1} cannot be easily consumed by the favorably signed terms on the left-hand side and additional steps have to be taken. These steps, however, necessitate that we can estimate $\veps$ from below in order to manipulate the appearing exponents in a beneficial way. As we allow the initial concentration of the signal chemical to be zero at parts of the domain, we obviously cannot hope for a time-global positive lower bound, but we can exploit the smoothing action of the Neumann heat semigroup combined with the nonnegative signal production term in the second equation to derive the following pointwise estimate on time-intervals strictly away from the initial time. While the result below has been employed in chemotaxis settings before, we were unable to find a version directly compatible with our conditions and hence provide the proof for completeness. Closely related results can be found e.g. in  \cite{fujie2015,fujie_senba_nonlinearity2016}.
\begin{lem}\label{lem2}
Let $\tau>0$. Then there exists $L(\Omega,\tau)>0$ such that if $T>\tau$,  
$h\in C^0(\bom\times(0,T))$ with $h\ge0$, and 
$z\in C^{2,1}(\bom\times(0,T))$ 
is a nonnegative solution of
\bas
\left\{
\begin{array}{ll}
z_t=\Del z -z + h(x,t),
\qquad & x\in\Om, \ t\in(0,T), \\[1mm]
\frac{\pa z}{\pa\nu}=0
\qquad & x\in\pO, \ t\in(0,T), \\[1mm]
\end{array}\right.
\eas
then 
\be{2.1}
z(x,t)
\ge L(\Omega,\tau)\cdot\inf_{s\in(0,T)}\io h(\cdot,s)
\qquad \mbox{for all $x\in\Om$ and }t\in(\tau,T).
\ee
\end{lem}
\proof
  Let $G(x,y,t)$  denote the Neumann heat kernel on $\Om$ (e.g. \cite[p.23]{ito_diff}).  Then for any fixed $\tau\in(0,T)$, 
  the strict positivity and continuity of $G(x,y,\frac{\tau}{2})$ on $\bom\times\bom$
  (\cite[Theorem 10.1]{ito_diff}) ensure the existence of
  $c_1(\tau,\Omega)>0$ such that
  \bas
  G(x,y,\tfrac{\tau}{2})\ge c_1(\tau,\Omega)
  \qquad \mbox{for all $x,y \in\bom$}.
  \eas
  Using the semigroup property (\cite[Theorem 8.1]{ito_diff})
 \bas
 \io G(x,z,t-\tfrac{\tau}{2})\cdot G(z,y,\tfrac{\tau}{2})dz= G(x,y,t),
 \eas 
 which holds for all $x,y\in\bom$ and $t>\frac{\tau}{2}$,
 together with
  $\io G(x,y,\sigma)dy=1$ for all $x\in\bom$ and $\sigma>0$, we infer that
  \bas
  G(x,y,t)\ge c_1(\tau,\Omega)
   \qquad \mbox{for all $x, y\in\bom$ and } t>\frac{\tau}{2}.
  \eas
  Hence for any nonnegative $\vp\in C^0(\bom)$, we have
  \bea{2.2}
  (e^{t\Del} \vp) (x) 
  &=& \io G(x,y,t) \vp(y) dy\nn\\
  &\ge& c_1(\tau,\Omega)\io \vp(y) dy
  \qquad \mbox{for all $x\in\bom$ and } t>\frac{\tau}{2}.
  \eea
  We now pick an arbitrary $\delta\in(0,\tfrac{\tau}{2})$, so that $t-\tfrac{\tau}{2}>\delta$ for all $t>\tau$. Applying the variation–of–constants formula to $z$, and using the
  nonnegativity of both $z$ and $h$,
  we conclude from (\ref{2.2}) that 
  \bas
  z(\cdot,t) 
  &=& e^{(t-\delta)(\Del-1)}z(\cdot,\delta) + \int_{\delta}^{t}e^{(t-s)(\Del-1)}h(\cdot,s)ds\\
  &\ge& c_1(\tau,\Omega)
  \cdot\Big\{\int_{\delta}^{t-\frac{\tau}{2}}e^{-(t-s)}ds\Big\}
   \cdot \inf_{s\in(\delta,t-\frac{\tau}{2})}\io h(\cdot,s)\\
  &\ge& c_1(\tau,\Omega)\cdot(e^{-\frac{\tau}{2}}-e^{\delta-\tau})\cdot \inf_{s\in(0,T)}\io h(\cdot,s)
  \qquad\mbox{for all $x\in \Omega$, }\delta\in(0,\tfrac{\tau}{2})\mbox{ and }t\in (\tau,T)
  \eas
  Letting $\delta\searrow 0$ and setting $L(\tau,\Omega):=  c_1(\tau,\Omega)\cdot(e^{-\frac{\tau}{2}}-e^{-\tau})$, we obtain (\ref{2.1}).
  \qed
A helpful result for the upcoming calculations is the following generalization of the Gagliardo--Nirenberg inequality (\cite[Lemma A.5]{taowin_jde2014}).
\begin{lem}\label{lem-tw2014}
Let $\Om\subset \mathbb{R}^2$ be a bounded domain with smooth boundary, and let $q\in(1,\infty)$, $r\in(0,q)$ and $\alpha>0$. Then there exists $C>0$ such that for each $\eta>0$ one can pick $C=C(\eta)>0$ with the property that 
\begin{align*}
\|\rho\|_{L^q(\Om)}^q\leq \eta\big\|\nabla\rho\big\|_{L^2(\Om)}^{q-r}\big\|\rho|\ln|\rho||^\alpha\big\|_{L^r(\Om)}^r+C\|\rho\|_{L^r(\Om)}^q+C
\end{align*}
holds for all $\rho\in W^{1,2}(\Om)$.
\end{lem}

When applied for $\alpha=\frac{p}{2}$ and $\rho=\varphi^{\frac{p+1}{2}}\psi^\frac12$ the inequality of Lemma~\ref{lem-tw2014} can be re-expressed in the following form, by drawing on the pointwise lower bound implied by Lemma~\ref{lem2}.
\begin{lem}\label{lem4}
Let $p>0$ and $\Gamma>0$. Then for every $\eta>0$ there exists $C=C(p,\Gamma,\eta)>0$ such that for all 
$\vp\in C^1(\bom)$ and $\psi\in C^1(\bom)$ with $\vp>0$ and $\psi\ge\Gamma$ in $\bom$,
\bea{4.1}
\io \vp^{p+2}\psi^{\frac{p+2}{p+1}}
&\le& \eta\cdot\bigg\{\io \vp\ln\vp + \io \vp + \io \vp\psi +|\Omega| \bigg\} \cdot
\bigg\{\io \vp^{p-1}\psi|\na \vp|^2
+ \io \vp^{p+1}\psi|\na \psi|^2\bigg\}\nn\\
& &+ C \Big\{\io \vp\psi\Big\}^{p+2} + C.
\eea
\end{lem}

\proof
 Fix $\eta>0$ and define
   \bas
  \eta_1= \eta_1(p,\Gamma,\eta):= \frac{\eta}{\max\Big\{\frac{p+1}{2}, \frac{3}{2}, \frac{|\ln \Gamma|}{2}\Big\}\cdot\max\Big\{\frac{(p+1)^2}{2},\frac{1}{2\Gamma^2}\Big\}}.
   \eas
   Employing Lemma~\ref{lem-tw2014} for $q=\frac{2p+4}{p+1}$, $r=\frac{2}{p+1}$, $ \alpha=\frac{p}{2}$ and $\rho=\vp^\frac{p+1}{2}\psi^\frac{1}{2}$, we find $c_1=c_1(p,\Gamma,\eta)>0$ such that
   \be{4.2}
   \big\|\rho\big\|_{L^\frac{2p+4}{p+1}(\Om)}^\frac{2p+4}{p+1}
   \le \eta_1\big\|\na\rho\big\|_{L^2(\Om)}^2 \big\|\rho|\ln\rho|^\frac{p}{2}\big\|_{ L^\frac{2}{p+1}(\Om)}^\frac{2}{p+1}
   +c_1\big\|\rho\big\|_{L^\frac{2}{p+1}(\Om)}^\frac{2p+4}{p+1}+c_1.
   \ee
  We are left with estimating the terms appearing on the right-hand side of \eqref{4.2} and start with the norm containing the logarithm.
  Since $\psi\ge\Gamma$ in $\overline{\Omega}$, we have $|\ln\psi|\le \psi+|\ln \Gamma|$.
  Using Young’s inequality, we obtain
\begin{align*}\
\big\|\rho|\ln \rho|^\frac{p}{2}\big\|_{{L^\frac{2}{p+1}(\Om)}}^\frac{2}{p+1}=\ &\io \vp\psi^\frac{1}{p+1}|\ln (\vp^{\frac{p+1}{2}}\psi^\frac{1}{2})|^\frac{p}{p+1}\nn\\
    \le\ & \io \vp|\ln (\vp^{\frac{p+1}{2}}\psi^\frac{1}{2})| + \io \vp\psi\nn\\
    =\ & \io \vp \Big|\frac{p+1}{2}\ln\vp + \frac{1}{2}\ln \psi\Big| + \io \vp\psi\nn\\
    \le\ & \frac{p+1}{2}\cdot\bigg\{\int_{\{\vp>1\}} \vp\ln\vp - \int_{\{\vp\le1\}} \vp\ln\vp\bigg\}
    + \frac{1}{2}\io \{|\ln \Gamma|+\psi\}\cdot\vp
    + \io \vp\psi\nn\\
=\ & \frac{p+1}{2}\int_{\{\vp>1\}}\vp\ln\vp-\frac{p+1}{2}\int_{\{\vp\le 1\}}\vp\ln\vp+\frac{|\ln\Gamma|}{2}\io\vp+\frac{3}{2}\io\vp\psi \nn,\\
\le\ & c_2\cdot\bigg\{\int_{\{\vp>1\}}\vp\ln\vp+\io\vp+\io\vp\psi\bigg\}-\frac{p+1}{2}\int_{\{\vp\le 1\}}\vp\ln\vp \nn
\end{align*}
where $c_2=c_2(p,\Gamma):= \max\Big\{\frac{p+1}{2}, \frac{3}{2}, \frac{|\ln \Gamma|}{2}\Big\}$.
Accordingly, using $|s\ln s|\le 1/e\le 1/2$ for $s\in(0,1]$ and $\frac{p+1}{2}\leq c_2$ yields that
\begin{align}\label{4.3}
\big\|\rho|\ln \rho|^\frac{p}{2}\big\|_{{L^\frac{2}{p+1}(\Om)}}^\frac{2}{p+1}
\le\ & c_2\cdot\bigg\{\io\vp\ln\vp+\io\vp+\io\vp\psi\bigg\}-\Big(c_2+\frac{p+1}{2}\Big)\int_{\{\vp\le 1\}}\vp\ln\vp \nn\\
\le\ & c_2\cdot\bigg\{\io\vp\ln\vp+\io\vp+\io\vp\psi\bigg\}+c_2|\Omega| \nn\\
    \le\ & c_2\cdot\bigg\{\io \vp\ln\vp + \io \vp + \io \vp\psi +|\Omega|  \bigg\}.
\end{align}
 Next, again using $\psi\ge\Gamma$, we obtain
\begin{align}\label{4.4}
\big\|\na \rho\big\|_{L^2(\Om)}^2&= \big\|\na(\vp^{\frac{p+1}{2}}\psi^\frac{1}{2})\big\|_{L^2(\Om)}^2\nn\\
 &= \bigg\|\frac{p+1}{2}\vp^\frac{p-1}{2}\psi^\frac{1}{2}\na\vp
 +\frac{1}{2}\vp^\frac{p+1}{2}\psi^{-\frac{1}{2}}\na\psi\bigg\|_{L^2(\Om)}^2\nn\\
 &\le \frac{(p+1)^2}{2}\io \vp^{p-1}\psi|\na\vp|^2
 +\frac{1}{2}\io \vp^{p+1}\psi^{-1}|\na\psi|^2\nn\\
 &\le \frac{(p+1)^2}{2}\io \vp^{p-1}\psi|\na\vp|^2
 + \frac{1}{2\Gamma^2} \io \vp^{p+1}\psi|\na\psi|^2\nn\\
 &\le c_3\cdot\bigg\{\io \vp^{p-1}\psi|\na\vp|^2
 + \io \vp^{p+1}\psi|\na\psi|^2\bigg\}
\end{align}
 with
 $c_3=c_3(p,\Gamma):= \max\Big\{\frac{(p+1)^2}{2}, \frac{1}{2\Gamma^2}\Big\}$ and, similarly,
  \bea{4.5}
 \big\|\rho\big\|_{L^{\frac{2}{p+1}}(\Om)}^{\frac{2p+4}{p+1}}=\|\vp^{\frac{p+1}{2}}\psi^\frac{1}{2}\|_{L^\frac{2}{p+1}(\Om)}^\frac{2(p+2)}{p+1}
 = \Big\{\io \vp\psi^\frac{1}{p+1}\Big\}^{p+2}\le \Gamma^{-\frac{p(p+2)}{p+1}}\Big\{\io \vp\psi\Big\}^{p+2}.
 \eea
Since
 \be{4.6}
  \|\vp^{\frac{p+1}{2}}\psi^\frac{1}{2}\|_{L^\frac{2(p+2)}{p+1}(\Om)}^\frac{2(p+2)}{p+1}
  = \io \vp^{p+2}\psi^{\frac{p+2}{p+1}},
 \ee
we find from combining \eqref{4.2}--\eqref{4.5}, and using the definitions of
 $\eta_1$, $c_1$, $c_2$ and $c_3$, that
 \bas
 \io \vp^{p+2}\psi^{\frac{p+2}{p+1}}
 &\le& \eta_1\cdot c_2c_3\cdot
 \bigg\{\io \vp\ln\vp + \io \vp + \io \vp\psi +|\Omega|  \bigg\}
 \cdot \bigg\{\io \vp^{p-1}\psi|\na\vp|^2
  + \io \vp^{p+1}\psi|\na\psi|^2\bigg\}\\
  & &+ c_1\Gamma^{-\frac{p(p+2)}{p+1}}\Big\{\io \vp\psi\Big\}^{p+2}
  +c_1\\
  &\le& \eta\cdot\bigg\{\io \vp\ln\vp + \io \vp + \io \vp\psi +|\Omega| \bigg\} \cdot
  \bigg\{\io \vp^{p-1}\psi|\na \vp|^2
  + \io \vp^{p+1}\psi|\na \psi|^2\bigg\}\\
  & &+ C \Big\{\io \vp\psi\Big\}^{p+2}
  +C
 \eas
 with $C= C(p,\Gamma,\eta):= c_1\cdot(1+\Gamma^{-\frac{p(p+2)}{p+1}})>0$ as claimed.
 \qed
These  preparations at hand, we can now go back to the differential inequality from Lemma~\ref{lem5} to derive new higher order bounds for both $\ueps$ and $\veps$.
\begin{lem}\label{lem6}
Let $p\in(1,\infty)$ and $\tau>0$. There exists $C(p,\tau)>0$ 
such that
\be{6.1}
\|\ueps(\cdot,t)\|_{L^p(\Om)}\le C(p,\tau)
\qquad \mbox{for all $t\in(\tau,\tme)$},
\ee
and 
\be{6.2}
\|\na\veps(\cdot,t)\|_{L^{2p+2}(\Om)}\le C(p,\tau)
\qquad \mbox{for all $t\in(\tau,\tme)$}.
\ee
%
\end{lem}
\proof
    From (\ref{mass}) and Lemma \ref{lem1}, we may fix $c_1>0$ such that
    \be{6.4}
    \io \ueps(\cdot,t)
    +\io \ueps(\cdot,t)\ln\ueps(\cdot,t)
    +\io \ueps(\cdot,t)\veps(\cdot,t)
    +\io \veps^2(\cdot,t)
    +\io|\na\veps(\cdot,t)|^2+|\Omega|
    \le c_1.
    \ee
    for all $t\in(0,\tme)$ .
    Moreover, Lemma~\ref{lem2} ensures that for each fixed $\tau>0$, there exists $c_2=c_2(\tau)>0$ fulfilling 
    \be{6.5}
    \veps(x,t)\ge c_2
    \qquad \mbox{for all $x\in\Om$ and $t\in\big(\frac{\tau}{2},\tme\big)$ and } \eps \in (0, \epsr).
    \ee
    
    \medskip
    The Gagliardo-Nirenberg inequality yields a constant $c_3=c_3(p)$ with the property that
    \be{6.6}
    \|\vp\|_{L^{\frac{2p+4}{p+1}}(\Om)}^\frac{2p+4}{p+1}
    \le c_3\|\na\vp\|_{L^2(\Om)}^2
    \|\vp\|_{L^\frac{2}{p+1}(\Om)}^\frac{2}{p+1}
    + c_3 \|\vp\|_{L^\frac{2}{p+1}(\Om)}^\frac{2p+4}{p+1}
    \qquad \mbox{for all } \vp\in W^{1,2}(\Om).
    \ee
    Invoking Lemma \ref{lem5}, followed by Young’s inequality, we find
    constants $c_4=c_4(p)>0$ and $c_5=c_5(p)>0$ such that
    \bea{6.7}
    & &\hs{-15mm}
    \frac{d}{dt}\bigg\{\io \ueps^p
    + \io |\na\veps|^{2p+2}\bigg\}
    +\frac{p(p-1)}{2}\io \ueps^{p-1}\veps |\na\ueps|^2
    +\frac{2p}{p+1}\io |\na|\na\veps|^{p+1}|^2\nn\\
    &\le&\frac{p(p-1)}{2}\io \ueps^{p+1} \veps |\na\veps|^2
        + c_4\io \ueps^2|\na\veps|^{2p} + c_4.\nn\\
    &\le& \frac{p(p-1)}{2}\io \ueps^{p+1} \veps |\na\veps|^2
    +  \frac{p}{2(p+1)c_1c_3}\io |\na\veps|^{2p+4}
    + c_5\io \ueps^{p+2}
    +c_5
    \eea
for all $t\in(0,\tme)$.
  
  \medskip
  
    To estimate the first term on the right-hand side, we use Young's inequality and obtain
    $c_6=c_6(p)$ such that 
   \be{6.12}
   \frac{p(p-1)}{2}\io \ueps^{p+1} \veps |\na\veps|^2
      \le \frac{p}{4(p+1)c_1c_3}\io |\na\veps|^{2p+4}
      + c_6\io \ueps^{p+2}\veps^{\frac{p+2}{p+1}}
   \ee
   for all $t\in(\frac{\tau}{2},\tme)$. Here, in light of \eqref{6.5} we can employ Lemma~\ref{lem4} with $\varphi=u$, $\psi=v$ and $\eta=\eta_1=\eta_1(p)= \frac{1}{c_1 c_6}\min\{\frac{p(p-1)}{8}, \frac{p-1}{2}\}$
to find upon combining with \eqref{6.4} that there is $c_7=c_7(p,\tau)>0$ such that
   \begin{align*}
   c_6\io \ueps^{p+2}\veps^{\frac{p+2}{p+1}}
   &\le \eta_1 c_1 c_6\cdot\Big\{
   \io \ueps^{p-1}\veps |\na\ueps|^2
   +\io \ueps^{p+1} \veps |\na\veps|^2\Big\}
  + c_7\\
  &\le  \frac{p(p-1)}{8}  \io \ueps^{p-1}\veps |\na\ueps|^2
  + \frac{p(p-1)}{4} \io \ueps^{p+1} \veps |\na\veps|^2
  + c_7
   \end{align*}
for all $t\in(\frac{\tau}{2},\tme)$.
   Inserting this into (\ref{6.12}) and rearranging yields
   \begin{align}\label{6.8}
    \frac{p(p-1)}{2} \io \ueps^{p+1} \veps |\na\veps|^2
   \le  \frac{p}{2(p+1)c_1c_3}\io |\na\veps|^{2p+4}
   + \frac{p(p-1)}{4}  \io \ueps^{p-1}\veps |\na\ueps|^2
   + 2c_7
   \end{align}
for all $t\in(\frac{\tau}{2},\tme)$.
  \medskip
  
  Next, in order to control the second gradient integral on the right of \eqref{6.7}
  and in the first term on the right-hand side of \eqref{6.8}, we combine
  (\ref{6.6}) with (\ref{6.4}) and obtain
   \begin{align}\label{6.9}
  \frac{p}{(p+1)c_1c_3}\io \big|\na\veps\big|^{2p+4}
  &= \frac{p}{(p+1)c_1c_3} \big\||\na\veps|^{p+1}\big\|_{L^{\frac{2p+4}{p+1}}(\Om)}^\frac{2p+4}{p+1}\nn\\
  &\le \frac{p}{(p+1)c_1}\cdot
   \bigg\{\big\|\na|\na\veps|^{p+1}\big\|_{L^2(\Om)}^2
      \big\||\na\veps|^{p+1}\big\|_{L^\frac{2}{p+1}(\Om)}^\frac{2}{p+1}
      +\big\||\na\veps|^{p+1}\big\|_{L^\frac{2}{p+1}(\Om)}^\frac{2p+4}{p+1}\bigg\}\nn\\
  &\le \frac{p}{p+1}\io \big|\na|\na\veps|^{p+1}\big|^2
  + c_8
  \end{align}
  for all $t\in (\frac{\tau}{2},\tme)$, with $c_8=c_8(p)>0$.
  
  \medskip
  
  It remains to control the integral without gradient contribution on the right-hand side of (\ref{6.7}). We invoke Lemma~\ref{lem-tw2014} for $\rho=\ueps^\frac{p+1}{2}$, $q=\frac{2p+4}{p+1}$, $r=\frac{2}{p+1}$, $\alpha=\frac{p+1}{2}$ and $\eta=\eta_2=\eta_2(p,\tau):=\frac{p(p-1)c_2}{c_5 (p+1)^3(2c_1+\frac{2}{e}|\Omega|)}$ to find $c_9=c_9(p,\tau)>0$ such that 
  \begin{align}\label{6.95}
  c_5\io \ueps^{p+2}
  &= c_5\Big\|\ueps^\frac{p+1}{2}\Big\|_{L^\frac{2p+4}{p+1}}^\frac{2p+4}{p+1}\nn\\
  &\le \eta_2\cdot c_5\Big\|\na\ueps^\frac{p+1}{2}\Big\|_{L^2(\Om)}^2
  \cdot \Big\|\ueps^{\frac{p+1}{2}}|\ln\ueps^\frac{p+1}{2}|^{\frac{p+1}{2}}\Big\|_{L^\frac{2}{p+1}(\Om)}^{\frac{2}{p+1}}
  + c_5c_9\Big\|\ueps^{\frac{p+1}{2}}\Big\|_{L^\frac{2}{p+1}(\Om)}^\frac{2p+4}{p+1}+ c_5c_9
  \end{align}
  for all $t\in(\frac{\tau}{2},\tme)$. Herein, we can use \eqref{6.4} to estimate
\begin{align*}
\Big\|\ueps^{\frac{p+1}{2}}|\ln\ueps^\frac{p+1}{2}|^\frac{p+1}{2}\Big\|_{L^\frac{2}{p+1}(\Om)}^\frac{2}{p+1}&=\io \ueps\big|\ln\ueps^{\frac{p+1}{2}}\big|\\
&=\frac{p+1}{2}\io\ueps\big|\ln\ueps\big|\\
&\leq \frac{p+1}{2}\Big(-2\int\limits_{\{\ueps<1\}}\ueps\ln\ueps+\io\ueps\ln\ueps\Big)\\
&\leq\frac{p+1}{2}\Big(\frac{2}{e}|\Omega|+c_1\Big)
\end{align*} 
for all $t\in(\frac{\tau}{2},\tme)$ and infer from \eqref{6.95} that
  \begin{align}\label{6.10}
  c_5\io \ueps^{p+2}
  &\le \eta_2\cdot c_5\cdot\frac{(p+1)}{2}\Big(c_1+\frac{2}{e}|\Omega|\Big)
 \io|\na\ueps^\frac{p+1}{2}|^2
  +c_{10}\nn\\
  &\le \frac{p(p-1)c_2}{2(p+1)^2}\io|\na\ueps^\frac{p+1}{2}|^2
  + c_{10}
  \qquad \mbox{for all $t\in\big(\frac{\tau}{2},\tme\big)$},
  \end{align}
  with $c_{10}=c_{10}(p,\tau):=c_5c_{9}\cdot
  \{1+ c_1^{p+2}\}>0$.
  
  \medskip
  
  Collecting (\ref{6.7})-(\ref{6.10}), we conclude that
  for all $t\in\big(\frac{\tau}{2}, \tme\big)$,
  \begin{align*}
       &\frac{d}{dt}\bigg\{\io \ueps^p
      + \io |\na\veps|^{2p+2}\bigg\}
      +\frac{p(p-1)}{2}\io \ueps^{p-1}\veps |\na\ueps|^2
      +\frac{2p}{p+1}\io |\na|\na\veps|^{p+1}|^2\nn\\
      \le\ & \frac{p}{p+1}\io |\na|\na\veps|^{p+1}|^2
        +\frac{p(p-1)}{4}  \io \ueps^{p-1}\veps |\na\ueps|^2
          +\frac{p(p-1)c_2}{2(p+1)^2}\io|\na\ueps^\frac{p+1}{2}|^2
          +c_{11}
  \end{align*}
    with $c_{11}=c_{11}(p,\tau):=2c_7+c_8+ c_{10}$. Hence, rewriting terms and using \eqref{6.5} we obtain
    \begin{align}\label{6.11}
    \frac{d}{dt}\bigg\{\io\ueps^p+\io|\na\veps|^{2p+2}\bigg\}+ \frac{p(p-1)c_2}{2(p+1)^2}\io|\na\ueps^\frac{p+1}{2}|^2+\frac{p}{p+1}\io|\na|\na \veps|^{p+1}|^2\leq c_{11}
    \end{align}
    for all $t\in(\frac{\tau}{2},\tme)$.
    
    \medskip
    
   We now derive from (\ref{6.11}) a superlinear damped ODI for
   \bas
   \yeps(t):= \io \ueps^p(\cdot,t)
   + \io |\na\veps(\cdot,t)|^{2p+2},
   \qquad t\in\Big(\frac{\tau}{2},\tme\Big)\ ,\ \eps\in(0,\epsr).
   \eas
   For $t\in\big(\frac{\tau}{2}, \tme\big)$, we have
    \bas
    \yeps^{\frac{p+1}{p}}(t)
    &=&\Big\{\io \ueps^p(\cdot,t)
        + \io |\na\veps(\cdot,t)|^{2p+2}\Big\}^{\frac{p+1}{p}}\\
    &\le& 2^{\frac{p+1}{p}}\Big\{\io \ueps^p(\cdot,t)\Big\}^{\frac{p+1}{p}}
    + 2^{\frac{p+1}{p}}\Big\{\io |\na\veps(\cdot,t)|^{2p+2}\Big\}^{\frac{p+1}{p}}\\
    &\le& 2^{\frac{p+1}{p}}\Big\{\io \ueps^p(\cdot,t)\Big\}^{\frac{p+1}{p-1}}
        + 2^{\frac{p+1}{p}}\Big\{\io |\na\veps(\cdot,t)|^{2p+2}\Big\}^{\frac{p+1}{p}}
        +2^{\frac{p+1}{p}}
    \eas   
   Here, again using the Gagliardo--Nirenberg inequality and (\ref{6.4}), we can find positive constants $c_{12},\dots, c_{15}$ depending only on $p$ such that
  \bas
  2^{\frac{p+1}{p}}\Big\{\io \ueps^p(\cdot,t)\Big\}^{\frac{p+1}{p-1}}
  &=& 2^{\frac{p+1}{p}} \big\|\ueps^\frac{p+1}{2}\big\|_{L^\frac{2p}{p+1}(\Om)}^\frac{2p}{p-1}\\
  &\le& c_{12} \big\|\na\ueps^\frac{p+1}{2}\big\|_{L^2(\Om)}^2
  \cdot \big\|\ueps^\frac{p+1}{2}\big\|_{L^\frac{2}{p+1}(\Om)}^\frac{2}{p-1}
  + c_{12} \big\|\ueps^\frac{p+1}{2}\big\|_{L^\frac{2}{p+1}(\Om)}^\frac{2p}{p-1}\\
  &\le& c_{13} \big\|\na\ueps^\frac{p+1}{2} \big\|_{L^2(\Om)}^2+ c_{13}
  \eas
  and
  \bas
   2^{\frac{p+1}{p}}\Big\{\io |\na\veps(\cdot,t)|^{2p+2}\Big\}^{\frac{p+1}{p}}
   &=&  2^{\frac{p+1}{p}} \big\||\na\veps|^{p+1}\big\|_{L^2(\Om)}^{\frac{2(p+1)}{p}}\\
   &\le& c_{14}\big\|\na|\na\veps|^{p+1}\big\|_{L^2(\Om)}^2\cdot
    \big\||\na\veps|^{p+1}\big\|_{L^\frac{2}{p+1}(\Om)}^{\frac{2}{p}}
    + c_{14} \big\||\na\veps|^{p+1}\big\|_{L^\frac{2}{p+1}(\Om)}^{\frac{2(p+1)}{p}}\\
    &\le& c_{15}\big\|\na|\na\veps|^{p+1}\big\|_{L^2(\Om)}^2
    + c_{15}.
  \eas
  Thus, using (\ref{6.11}) and (\ref{6.5}), if we set
  $c_{16}:=c_{16}(p,\tau)=\min\{ \frac{p(p-1)c_2}{2(p+1)^2}, \frac{p}{p+1}\}$,
  we obtain, for all $t\in (\frac{\tau}{2},\tme)$,
  \bas
  \frac{c_{16}}{c_{13}+c_{15}}\cdot\yeps^{\frac{p+1}{p}}(t)
  -c_{17}
  \le {c_{16}}\io|\na\ueps^\frac{p+1}{2}|^2
  +c_{16}\io |\na|\na\veps|^{p+1}|^2
  \eas
for some $c_{17}=c_{17}(p,\tau)>0$, and hence (\ref{6.11}) yields
  \be{6.13}
  \yeps'(t)
  + \frac{c_{16}}{c_{13}+c_{15}}\cdot\yeps^{\frac{p+1}{p}}(t)
  \le c_{11}+c_{17}
  \ee 
  for all $t\in (\frac{\tau}{2},\tme)$ and $\eps\in(0,\epsr)$.
A comparison argument (e.g. \cite[Lemma 2.5]{taowin_3mas}) applied to
 (\ref{6.13}) shows that
\bas
\yeps(t)\le c_{18}
\qquad \mbox{$t\in [\tau,\tme)$,}
\eas
for some $c_{18}=c_{18}(p,\tau)>0$. In particular,
\bas
\|\ueps(\cdot,t)\|_{L^p(\Om)}^p
+ \|\na \veps(\cdot,t)\|_{L^{2p+2}(\Om)}^{2p+2}
\le c_{18}
\eas
for all $t\in(\tau,T_{\max,\varepsilon})$, which concludes the proof.
\qed

A direct consequence of the bounds obtained above is a corresponding bound of $\veps$ in $W^{1,\infty}(\Om)$.
\begin{lem}\label{lem17}
Let $\tau>0$. There exists $C(\tau)>0$ with the property that 
 \be{7.2}
  \|\veps(\cdot,t)\|_{W^{1, \infty}(\Om)}
  \le C(\tau).
 \qquad \mbox{for all $t\in(\tau,\tme)$}.
 \ee
\end{lem}
\proof
     First we recall that, for any fixed $p>2$,
     Lemma 1.3 of \cite{win_jde2010} ensures the existence of $c_1(\tau)>0$ 
     such that
     \bas
     \|e^{\sigma\Del}\vp\|_{W^{1,\infty}(\Om)}
     \le c_1(\tau) \sigma^{-\frac{1}{2}-\frac{1}{p}} \|\vp\|_{L^p(\Om)}
     \qquad\mbox{for all $\vp\in C^0(\bom)$ and } \sigma\in\Big(0,\frac{\tau}{2}\Big].
     \eas
	Employing the variation-of-constant formula while drawing on \eqref{1.2} and \eqref{6.1}, we obtain $c_2(\tau)>0$ such that
    \bas
     \|\veps(\cdot,t)\|_{W^{1,\infty}(\Om)}
     &\le& \Big\| e^{\frac{\tau}{2}(\Del-1)}\veps(\cdot,t-\frac{\tau}{2})\Big\|_{W^{1,\infty}(\Om)}
     + \int_{t-\frac{\tau}{2}}^t \Big\| e^{(t-s)(\Del-1)}\ueps(\cdot,s)\Big\|_{W^{1,\infty}(\Om)}ds\\
     &\le& c_1(\tau)\cdot\Big(\frac{\tau}{2}\Big)^{-\frac{1}{2}-\frac{1}{p}}
     \cdot\Big\|\veps(\cdot,t-\frac{\tau}{2})\Big\|_{L^p(\Om)}\nn\\
     & &+ c_1(\tau) \int_{t-\frac{\tau}{2}}^t e^{-(t-s)}(t-s)^{-\frac{1}{2}-\frac{1}{p}}
      \|\ueps(\cdot,s)\|_{L^p(\Om)}ds\\
      &\le& c_2(\tau)
      \qquad\mbox{for all }\eps \in(0,\epsr)\mbox{ and }t>\tau,
    \eas
as claimed.
   \qed

%
%
%
%
%
%
%
We can now draw on a Moser-type iteration to conclude the global existence of our approximate solution and the uniform boundedness of $\ueps$ away from the initial time.
\begin{lem}\label{lem7}
For each $\eps\in(0,\epsr)$, we have $\tme=\infty$. Moreover, for each $\tau>0$, 
there exists $C=C(\tau)>0$ such that 
 \be{7.1}
 \|\ueps(\cdot,t)\|_{L^\infty(\Om)}
 \le C(\tau)
 \qquad \mbox{for all $\ t>\tau\ $ and }\ \eps\in(0,\epsr).
 \ee
\end{lem}
\proof
For $\eps\in(0,\epsr)$ we let $\tau_\eps:=\min\{1,\tfrac{1}{2}\tme\}$ and choose $\zeta_{\eps}\in C^\infty([0,\infty))$ such that
     $\zeta_{\eps}=0$ on $[0,\frac{\tau_\eps}{2}]$ and $\zeta_{\eps}=1$ on $[\tau_\eps,\infty)$. Set
    \bas
    \zeps(x,t)=\zeta_{\eps}(t)\ueps(x,t)
    \qquad \mbox{for all }x\in\bom\mbox{ and } t\in(0,\tme),
    \eas
    so that
    \be{7.4}
    \zeps\le \ueps
    \qquad \mbox{ for all $x\in\bom$ and  $t\in(0,\tme)$.}
    \ee 
    From (\ref{0eps}) we obtain
    \be{7.3}
    \left\{	
     \begin{array}{ll}
     z_{\eps t}= \na\cdot(\ueps\veps\na\zeps)
     -\na\cdot(\ueps\zeps\veps\na\veps)+\zeta_{\eps}'\ueps,  
      \qquad & x\in\Om, \ t>\frac{\tau_\eps}{2}, \\[1mm]
     \frac{\pa\zeps}{\pa\nu}=0,
      	\qquad & x\in\pO, \ t>\frac{\tau_\eps}{2}, \\[1mm]
      	\zeps\big(x,\frac{\tau_\eps}{2}\big)=0
      	\qquad & x\in\Om.
      \end{array}
    \right.
    \ee
   Using (\ref{7.4}) and Lemma \ref{lem2}, we infer the existence of $c_1(\tau_\eps)>0$ such that
   \bas
   \ueps\veps
   \ge c_1(\tau_\eps)\zeps
     \qquad \mbox{for all $x\in\bom$ and $t\in(0,\tme)$.}
   \eas
   In addition, by (\ref{1.2}) in Lemma \ref{lem1} and Lemmas \ref{lem6} and \ref{lem17}, for any fixed $p>2$, we may find
    $c_2(p,\tau_\eps)>0$ with the property that
    \bas
    \sup_{t\in\big(\frac{\tau_\eps}{2},\tme\big)}\Big\{
    \|\zeps(\cdot,t)\|_{L^p(\Om)}
    + \|\aeps(\cdot,t)\|_{L^p(\Om)}
    +\|\beps(\cdot,t)\|_{L^p(\Om)}+\|\zeta_{\eps}'\ueps\|_{L^p(\Om)}\Big\}
    \le c_2(p,\tau_\eps)
    \eas
    where $\aeps(x,t):=\ueps(x,t)\veps(x,t)$ 
    and $\beps:=\ueps(x,t)\zeps(x,t)\veps(x,t)\na\veps(x,t)$ 
    for all $x\in\bom$, $t\in(0,\tme)$. 
    Consequently, standard Moser-type iteration (e.g. \cite{taowin_subcrit}) provides $c_3(\tau_\eps)>0$
    such that
   \bas
   \|\zeps(\cdot,t)\|_{L^\infty(\Om)}
   \le c_3(\tau_\eps)
   \qquad \mbox{for all $t\in\big(\frac{\tau_\eps}{2},\tme\big)$}
   \eas
   and therefore,
   \bas
     \|\ueps(\cdot,t)\|_{L^\infty(\Om)}
     \le c_3(\tau_\eps)
     \qquad \mbox{for all $t\in(\tau_\eps,\tme)$,}
     \eas
  which together with (\ref{ext}) shows $\tme=\infty$. Since $\eps\in(0,\epsr)$ was arbitrary, we have $\tme=\infty$ for all $\eps\in(0,\epsr)$ and a posteriori conclude \eqref{7.1} by repeating the steps for arbitrary $\tau>0$ instead of $\tau_\eps$ as now with $\tme=\infty$ the condition $\tau\in(0,\tme)$ is always satisfied for all $\eps\in(0,\epsr)$.
\qed
Utilizing well-known Hölder estimates up to the boundary we can further refine the known a priori information for both components.
  
\begin{lem}\label{lem8}
Let $\tau\in(0,1)$ and $T>1$. Then there exist $\theta=\theta(\tau,T)\in(0,1)$ 
and $C(\tau,T)>0$ such that
\bas
\|\ueps\|_{C^{\theta,\frac{\theta}{2}}(\bom\times[\tau,T])}
+ \|\veps\|_{C^{2+\theta,1+\frac{\theta}{2}}(\bom\times[\tau,T])}
\le C(\tau,T)
\mbox{\qquad for all  $\eps\in(0,\epsr)$}.
\eas
\end{lem}
\proof
   Lemmas~\ref{lem2}, \ref{lem17}, and \ref{lem7} provide bounds for
   $\ueps$, $\veps$, $\veps^{-1}$, and $|\na\veps|$ on
   $\bom\times[\frac{\tau}{4},2T]$, uniformly for $\eps\in(0,\epsr)$.
   The structural conditions required by
   \cite[Theorem~1.3]{porzio_vespri} can be verified as in the proof of
   \cite[Lemma~6.4]{li_win_aa2025}: the first equation in (\ref{0eps})
   has the same diffusion and taxis terms and homogeneous no-flux condition,
   while the reaction term $b_\eps$ in \cite[Lemma~6.4]{li_win_aa2025}
   is zero for (\ref{0eps}).
   After shifting the starting time to $\frac{\tau}{4}$,
   \cite[Theorem~1.3]{porzio_vespri}, which covers the Neumann boundary,
   yields $\theta_1=\theta_1(\tau,T)\in(0,1)$ and $c_1(\tau,T)>0$ such that
   \be{8.01}
   \|\ueps\|_{C^{\theta_1,\frac{\theta_1}{2}}(\bom\times[\frac{\tau}{2},2T])}
   \le c_1(\tau,T)
\mbox{\qquad for all $\eps\in(0,\epsr)$}.
   \ee
   The exponent $\theta_1$ and the constant $c_1$ depend on $\tau$, $T$,
   and the bounds for $\ueps$, $\veps$, $\veps^{-1}$, and $|\na\veps|$
   provided by Lemmas~\ref{lem2}, \ref{lem17}, and \ref{lem7},
   independently of $\eps$; (\ref{8.01}) does not require a uniform
   positive lower bound for $\ueps$.
   In view of the H\"older bound (\ref{8.01}) for $\ueps$ and
   the uniform bound for $\veps$ in (\ref{7.2}),
   parabolic Schauder estimates up to the boundary for the homogeneous
   Neumann problem \cite{LSU}, applied to the second equation in (\ref{0eps}),
   yield, on the smaller time interval $[\tau,T]$,
   $\theta_2=\theta_2(\tau,T)\in(0,1)$ and $c_2(\tau,T)>0$ such that
   \be{8.02}
   \|\veps\|_{C^{2+\theta_2,1+\frac{\theta_2}{2}}(\bom\times[\tau,T])}
   \le c_2(\tau,T)
\mbox{\qquad for all $\eps\in(0,\epsr)$}.
   \ee
   Combining (\ref{8.01}) and (\ref{8.02}) with
   $\theta=\min\{\theta_1,\theta_2\}$ proves the assertion.
   \qed
In order to pass some of the integrability information on to the limit functions we will establish in the next section, we additionally verify the uniform integrability of both $(|\nabla\veps|^2)_{\eps\in(0,\epsr)}$ and $(\ueps\veps)_{\eps\in(0,\epsr)}$.
\begin{cor}\label{cor9}
Let $T>0$. Then the families
\bas
(|\na\veps|^2)_{\eps\in(0,\epsr)}\ 
\mbox{ and }\ 
(\ueps\veps)_{\eps\in(0,\epsr)}
\qquad\mbox{are uniformly integrable over } \Om\times(0,T).
\eas
\end{cor}
\proof
   By Lemma \ref{lem1}, we can find $c_1>0$ such that
   \be{8.1}
   \io |\na\veps(\cdot,t)|^2
   \le c_1
   \qquad\mbox{for all $t>0$}.
   \ee
   Given $\eta>0$, we choose $\tau=\tau(\eta)>0$ sufficiently small such that
   \be{8.2}
   c_1\cdot\tau
   <\frac{\eta}{2}.
   \ee
   Applying Lemma \ref{lem17} yields the existence of $c_2=c_2(\eta)>0$ fulfilling
   \be{8.3}
   \|\na\veps(\cdot,t)\|_{L^\infty(\Om)}
   \le c_2
   \qquad\mbox{for all $t>\tau$}.
   \ee
  and thereafter we choose $\del=\del(\eta)>0$ small enough to satisfy 
  \be{8.4}
  \delta\cdot c_2^2
  \le \frac{\eta}{2}.
  \ee 
	Now, given an arbitrary measurable set $E\subset\Om\times(0,T)$ with $|E|<\delta$, we denote for any $t\in(0,T)$ the time-slice at time $t$ by
 \bas
 E(t):=\{x\in\Om\ |\ (x,t)\in E\}.
 \eas
 Then, using (\ref{8.1}) and (\ref{8.3}), we infer
 \bas
\iint_E |\na\veps|^2
&=& \int_0^\tau\int_{E(t)}|\na\veps|^2
+ \int_\tau^T\int_{E(t)}|\na\veps|^2\\
&\le& c_1\cdot\tau
+ c_2^2\int_0^T |E(t)|\\
&\le&  c_1\cdot\tau
+c_2^2\cdot\del,
 \eas
 which together with (\ref{8.2}) and (\ref{8.4}) entails
\bas
\iint_E |\na\veps|^2
\le \eta.
\eas 
In light of the arbitrariness of $E$, we therefore prove that
 $(|\na\veps|^2)_{\eps\in(0,1)}$ is uniformly integrable over $\Om\times(0,T)$.
 
 \medskip
 
Likewise, relying on Lemmas \ref{lem1} and \ref{lem7} once again, we can establish
the uniform integrability of $(\ueps\veps)_{\eps\in(0,\epsr)}$ over $\Om\times(0,T)$.
\qed
\setcounter{equation}{0}
\section{Construction of limit functions}\label{s5}
We are now in position to obtain limit functions suitable for our course by employing the Arzelà--Ascoli theorem.
\begin{lem}\label{lem11}
There exist a sequence $(\eps_j)_{j\in\N}\subset(0,\epsr)$ and nonnegative
 functions 
 \be{11.1}
 \left\{
 \begin{array}{l}
 u\in L^1_{loc}(\bom\times[0,\infty))
 \cap C^0_{loc}(\bom\times(0,\infty))
 \qquad\mbox{and}\\
 v\in L^2_{loc}([0,\infty);W^{1,2}(\Omega))
 \cap C^{2,1}_{loc}(\bom\times(0,\infty))
 \end{array}\right.
 \ee
 satisfying  $v>0$ on $\bom\times(0,\infty)$ and such that,  as $\eps=\eps_j\searrow 0$, 
  \begin{eqnarray}
    & &\ueps \to u
    \qquad\mbox{in } L^1_{loc}(\bom\times[0,\infty)),
    \label{11.2}\\
    & &
    \ueps \to u
        \qquad\mbox{in } 
      C^0_{loc}(\bom\times (0,\infty)),
      \label{11.7}\\
 	& & \veps\to v
 	\qquad \mbox{in } C^{2,1}_{loc}(\bom\times (0,\infty)),
 	\label{11.3} \\
 	& & \veps\to v
 	 \qquad \mbox{in } L^2_{loc}(\bom\times[0,\infty)),
 	 \label{11.8} \\  
 	& & \na\veps\to \na v
 	\qquad \mbox{in } L^2_{loc}(\bom\times [0,\infty)),
 	\label{11.4}\\
 	& &v_{\eps t} \wto v_t
 	\qquad \mbox{in } L^2(\bom\times (0,\infty))
 	\label{11.5}
 	\qquad \qquad \mbox{and}\\
 	& & \ueps\veps \to uv
 	 \qquad\mbox{in } L^1_{loc}(\bom\times[0,\infty)).
 	    \label{11.6}
   \end{eqnarray}
   Moreover,
   \be{11.10}
     \io u(\cdot,t)=\io u_0
     \qquad\mbox{for all }t>0,
     \ee
     and for any $p\in(1,\infty)$, one has
     \be{11.9}
     \lim_{t \searrow 0}\|v(\cdot,t)-v_0\|_{L^p(\Om)}=0.
     \ee
\end{lem}
\proof
   Lemma \ref{lem8} and the Arzel\`a--Ascoli theorem first give
   (\ref{11.7}) and (\ref{11.3}). 
    To verify (\ref{11.10}), we use
     (\ref{11.7}) and (\ref{mass})  to see that as $\eps=\eps_j\searrow0$,
      \bas
      \io u(\cdot,t)
      =\lim_{j\to\infty}\io u_{\eps_j}(\cdot,t)
      =\lim_{j\to\infty}\big(m+\eps_j|\Om|\big)
      =m=\io u_0
      \qquad\mbox{for all }t>0,
      \eas
     which proves (\ref{11.10}).
   Since (\ref{11.7}) provides pointwise convergence 
   of $\ueps$ on $\Om\times(0,\infty)$, and since $(\ueps)_{\eps\in(0,1)}$
   is uniformly integrable on $\Om\times(0,T)$ for all $T>0$ by (\ref{1.1}) in Lemma \ref{lem1},
   the Vitali convergence theorem applies and gives (\ref{11.2}). Next, 
  (\ref{1.2}) implies the uniform integrability of
   $(\veps^2)_{\eps\in(0,\epsr)}$ in $\Om\times(0,T)$ for every $T>0$,
   which, together with (\ref{11.3}) and the Vitali convergence theorem, yields (\ref{11.8}).
   Similarly, combining Corollary \ref{cor9} with (\ref{11.3}) and invoking the
   Vitali convergence theorem once more, we obtain (\ref{11.4}).
   Likewise, (\ref{1.3}) together with (\ref{11.7}), (\ref{11.3}), and Corollary \ref{cor9}
   yields (\ref{11.6}). Consequently, (\ref{1.4}) in conjunction with (\ref{11.3})
    furnishes (\ref{11.5}). The strict positivity of $v$ on $\bom\times(0,\infty)$ is a consequence of the fact that Lemma~\ref{lem2} entails that for each fixed $\tau>0$ there is $c(\tau)>0$ such that $\veps(x,t)\geq c(\tau)$ on $\Om\times(\frac{\tau}{2},\infty)$.
   
   \medskip
   
   Finally, employing the H\"older inequality and (\ref{1.4}) once more, we find
   $c_1>0$ such that
\begin{align*}
  \|\veps(\cdot,t)-v_{0\eps}\|_{L^2(\Om)}
  &\le \int_0^t\|v_{\eps t}(\cdot,s)\|_{L^2(\Om)}ds\\
  &\le c_1\cdot t^\frac{1}{2}
  \qquad\mbox{for all $t>0$ and }\eps\in(0,\epsr),
\end{align*}
which along with the Gagliardo--Nirenberg inequality, (\ref{init}), and
(\ref{1.2}), implies that for each fixed $p\in(2,\infty)$, there exist constants
 $c_2, c_3>0$ such that
\begin{align*}
 \|\veps(\cdot,t)
 -v_{0\eps}\|_{L^p(\Om)}
 &\le c_2\|\veps(\cdot,t)- v_{0\eps}\|^{\frac{p-2}{p}}_{W^{1,2}(\Om)}
 \cdot \|\veps(\cdot,t)
  -v_{0\eps}\|_{L^2(\Om)}^\frac{2}{p}\\
  &\le c_3\|\veps(\cdot,t)
    -v_{0\eps}\|_{L^2(\Om)}^\frac{2}{p}\\
  &\le c_1^\frac{2}{p}c_3\cdot t^\frac{1}{p}
    \qquad\mbox{for all $t>0$ and }\eps\in(0,\epsr).
\end{align*}
Combining this with (\ref{11.3}) and the convergence $v_{0\eps}\to v_0$
in $L^p(\Omega)$ as $\eps=\eps_j\searrow 0$, which follows from \eqref{init}, proves
\bas
\lim_{t \searrow 0}\|v(\cdot,t)-v_0\|_{L^p(\Om)}=0.
\eas
 The convergences (\ref{11.2})-(\ref{11.6}) now entail
 (\ref{11.1}). 
    \qed

With the bounds prepared in the earlier sections, we can replicate the approach of \cite[Lemma 3.1]{li_win_prsea2026} to augment the convergence properties of Lemma~\ref{lem11} by an additional strong convergence property for the gradient of the first solution component.
\begin{lem}\label{lem10}
Let $(u,v)$ and $(\eps_j)_{j\in\N}$ be as provided by Lemma \ref{lem11}.
Then, after passage to a further subsequence again denoted by $(\eps_j)_{j\in\N}$,
for all $\tau>0$ and any $T>\tau$ we have
\be{10.1}
 \na\ueps\to\na u
  \qquad\mbox{in $L^2(\Om\times (\tau,T))$ and a.e. in }
  \Om\times(0,\infty)
\ee
and
  \be{10.6}
  \sqrt{\veps}\na\ueps-\ueps\sqrt{\veps}\na\veps\wto
   \sqrt{v}\na u-u\sqrt{v}\na v
   \qquad\mbox{in }L^2(\Om\times(0,T))
  \ee
as $\eps=\eps_j\searrow 0$.
\end{lem}
\proof
 The argument establishing (\ref{10.1}) follows the method of \cite{li_win_prsea2026}, adapted to the present setting
 in which the bounds for $\veps, \na\veps$, and $\ueps$ are only available 
 away from the initial time and for $\eps\in(0,\epsr)$.
For fixed $T>\tau>0$, the estimates below are therefore derived on
 $(\frac{\tau}{2},T)$.

\medskip
\textbf{Claim 1.}
For every $\kappa>0$, $\tau>0$, and $T>\tau$ there exists $c_1(\kappa,\tau,T)>0$ with
  the property that
  \be{10.2}
  \int_{\frac{\tau}{2}}^{T}\io \ueps^{\kappa-1}|\na\ueps|^2
  \le c_1(\kappa,\tau,T)
  \qquad\mbox{for all }\eps\in(0,\epsr).
  \ee
  Indeed, this can be seen by arguing as in \cite[Lemma 3.1]{li_win_prsea2026}.
  Fixing $p\in(0,1)$ such that $p\leq \kappa$, we test the first equation in (2.1) by
  $u_\eps^{p-1}$ to obtain, after an integration by parts and Young's
   inequality,
  \bas
  \frac1p\frac{d}{dt}\int_\Omega u_\eps^p
  \ge  \frac{1-p}{2}\int_\Omega u_\eps^{p-1}v_\eps|\nabla u_\eps|^2
  - \frac{1-p}{2}\int_\Omega u_\eps^{p+1}v_\eps|\nabla v_\eps|^2.
  \eas
By Lemmas \ref{lem2}, \ref{lem17} and \ref{lem7}, we can find positive constants
$c_2(p,\tau,T)$ and $c_3(p,\tau,T)$ such that
\bas
\frac1p\frac{d}{dt}\int_\Omega u_\eps^p
\ge c_2(p,\tau,T)\int_\Omega u_\eps^{p-1}|\nabla u_\eps|^2
-c_3(p,\tau,T)
\qquad\mbox{for all $t\in(\frac{\tau}{2},T)$  and }\eps\in(0,\epsr).
\eas
Upon integration in time and using Lemma \ref{lem7}, this provides $c_4(p,\tau,T)>0$ such that
\bas
\int_{\frac{\tau}{2}}^T\int_\Omega
u_\eps^{p-1}|\nabla u_\eps|^2 
\le c_4(p,\tau,T)
\qquad\mbox{for all }\eps\in(0,\epsr).
\eas    
  Recalling our choice of $p$ we have $\kappa-p\geq0$ and rewriting
  \bas
   \int_{\frac{\tau}{2}}^{T}\io \ueps^{\kappa-1}|\na\ueps|^2
   =\int_{\frac{\tau}{2}}^{T}\io \ueps^{\kappa-p}
   \ueps^{p-1}|\na\ueps|^2\quad\text{for all }\eps\in(0,\epsr),
  \eas
hence entails \eqref{10.2} in view of Lemma~\ref{lem7}.
 
\medskip
\textbf{Claim 2.} 
For any $\kappa>0$ and $T>\tau>0$, one may choose $ c_5(\kappa,\tau,T)>0$
so that
\be{10.3}
\int_{\tau}^{T}\io \ueps^{\kappa+1}|D^2\ueps|^2
+ \int_{\tau}^{T}\io \Big|D^2\ueps^{\frac{\kappa+3}{2}}\Big|^2
\le  c_5(\kappa,\tau,T).
\ee
Here and below the bounds hold for all $\eps\in(0,\epsr)$.

Arguing as in Lemma 3.4 of \cite{li_win_prsea2026}, the
 first equation in (\ref{0eps}) can be rewritten in the form
  \bas
	u_{\eps t} = \ueps\veps\Del\ueps + \veps |\na\ueps|^2
	+ a_\eps(x,t) \ueps,
	\qquad (x,t)\in\Om\times (0,\infty),
  \eas
 with
\be{10.coefficient_a}
	a_\eps(x,t):=\na\ueps\cdot\na\veps
	- \ueps\veps \Del\veps - \ueps |\na\veps|^2
	- 2\veps\na\ueps\cdot\na\veps,
	\qquad (x,t)\in\Om\times (0,\infty).
  \ee
Using (\ref{10.2}) together with Lemmas \ref{lem17}, \ref{lem7}, 
and \ref{lem8}, we infer as in \cite{li_win_prsea2026} that there exists $c_6(\tau,T)>0$ such that
\be{10.4}
\int_{\frac{\tau}{2}}^T\io \aeps^2
\le c_6(\tau,T)
\qquad\mbox{for all }\eps\in(0,\epsr).
\ee

Furthermore, for each $\tau>0$, Lemmas \ref{lem2},
 \ref{lem17}, and \ref{lem7}, applied with $\frac{\tau}{2}$, furnish positive constants 
 $c_7(\tau)$ and $c_8(\tau)$ satisfying for all $T>\tau$,
\be{10.coefficient_bounds}
\ueps\le  c_7(\tau),
\quad\ 
c_8(\tau)\le \veps
\le c_7(\tau),
\quad\ \mbox{and}\quad\ 
|\na\veps|\le c_7(\tau)
\ee
 on $\Om\times(\frac{\tau}{2},T)$ for all $\eps\in(0,\epsr)$.
For the prescribed $\kappa>0$, fix $\gamma\in(0,1)$ with
$\gamma\le\kappa$. For $t\in(\frac{\tau}{2},T)$ and
$\eps\in(0,\epsr)$, introduce
\bas
\yeps(t):=\io \ueps^\gamma(\cdot,t)|\na\ueps(\cdot,t)|^2,
\eas
and, for a suitable constant $c_9(\gamma,\tau,T)>0$, define
\be{10.h_definition}
\heps(t):=c_9(\gamma,\tau,T)\io \ueps^{\gamma+1}(\cdot,t)
|D^2\ueps(\cdot,t)|^2
+c_9(\gamma,\tau,T)\io \ueps^{\gamma-1}(\cdot,t)
|\na\ueps(\cdot,t)|^4
\ee
For these quantities, we repeat the calculation in
equations~(3.15)--(3.27) of \cite{li_win_prsea2026}.
The argument carries over because it uses only the first equation
in (\ref{0eps}), the coefficient bounds
(\ref{10.coefficient_bounds}), and the space--time $L^2$ estimate
(\ref{10.4}) for $\aeps$, with the signal equation entering only
through these estimates and the logistic contribution $1-\ueps$
omitted from $\aeps$ as defined in (\ref{10.coefficient_a}).
Then there exist constants $c_{10}(\gamma,\tau,T)>0$ and
$c_{11}(\gamma,\tau,T)>0$ such that
\be{10.odi}
\yeps'(t)+c_{10}(\gamma,\tau,T)\yeps^2(t)+\frac14\heps(t)
\le c_{11}(\gamma,\tau,T)\io\aeps^2+c_{11}(\gamma,\tau,T)
\ee
for all $t\in(\frac{\tau}{2},T)$ and $\eps\in(0,\epsr)$. By (\ref{10.4}), applying the ODI comparison lemma
\cite[Lemma~3.3]{li_win_prsea2026} to (\ref{10.odi}) with starting time
$t_0=\frac{\tau}{2}$, we may find $c_{12}(\gamma,\tau,T)>0$ such that
\bas
\yeps(t)\le c_{12}(\gamma,\tau,T)
\qquad\mbox{for all $t\in[\tau,T)$ and }\eps\in(0,\epsr).
\eas
Integrating (\ref{10.odi}) from $\tau$ to $T$
therefore yields $c_{13}(\gamma,\tau,T)>0$ such that
\be{10.h_integral}
\int_{\tau}^{T}\heps(t)dt
\le c_{13}(\gamma,\tau,T)
\qquad\mbox{for all }\eps\in(0,\epsr).
\ee
Combining (\ref{10.h_definition}) and (\ref{10.h_integral}) gives
\be{10.weighted_gamma}
\int_{\tau}^{T}\io \ueps^{\gamma+1}(\cdot,t) |D^2 \ueps(\cdot,t)|^2 dt
+ \int_{\tau}^{T}\io \ueps^{\gamma-1}(\cdot,t) |\na\ueps(\cdot,t)|^4 dt
\le \frac{c_{13}(\gamma,\tau,T)}{c_9(\gamma,\tau,T)}.
\ee
Since $\gamma\le\kappa$, the bound $\ueps\le c_7(\tau)$ in (\ref{10.coefficient_bounds}) implies
$\ueps^{\kappa-\gamma}\le c_7(\tau)^{\kappa-\gamma}$.
Multiplying the integrands in (\ref{10.weighted_gamma}) by
$\ueps^{\kappa-\gamma}$ therefore yields a constant $c_{14}(\kappa,\tau,T)>0$ such that
\be{10.weighted_kappa}
\int_{\tau}^{T}\io \ueps^{\kappa+1}(\cdot,t) |D^2 \ueps(\cdot,t)|^2 dt
+ \int_{\tau}^{T}\io \ueps^{\kappa-1}(\cdot,t) |\na\ueps(\cdot,t)|^4 dt
\le c_{14}(\kappa,\tau,T).
\ee
Applying the proof of \cite[Corollary~3.5]{li_win_prsea2026}, with $\kappa$ in place of $\gamma$,
to (\ref{10.weighted_kappa}) proves (\ref{10.3}).

\medskip
\textbf{Claim 3.} 
Let $q>2$ and $\kappa>0$. For each $T>\tau>0$ there
exists $c_{13}(\kappa,\tau,T)>0$ with the property that
\be{10.5}
\int_{\frac{\tau}{2}}^T\|\pa_t\ueps^\kappa(\cdot,t)\|_{(W^{1,q}(\Om))^\star}dt
\le c_{13}(\kappa,\tau,T)
\qquad\mbox{for all }\eps\in(0,\epsr).
\ee

The testing argument in \cite[Lemma~3.6]{li_win_prsea2026}, with the logistic
terms omitted and restricted to $(\frac{\tau}{2},T)$, applies by
Claim~1 and Lemmas~\ref{lem17} and \ref{lem7}.

\medskip
Combining (\ref{10.2}), (\ref{10.3}), and (\ref{10.5}) and repeating
the arguments of \cite[Lemmas~3.7 and 3.8]{li_win_prsea2026} on each fixed
cylinder $\Om\times(\tau,T)$ gives, along the sequence $(\eps_j)_{j\in\N}$
provided by Lemma~\ref{lem11},
\be{10.strong_original}
\na u_{\eps_j}\to\na u
\qquad\mbox{in }L^2(\Om\times(\tau,T))
\quad\mbox{for every }0<\tau<T.
\ee
For each $n\ge2$, (\ref{10.strong_original}) on
$\Om\times(\frac1n,n)$ allows the extraction of a subsequence with
almost-everywhere gradient convergence on $\Om\times(\frac1n,n)$.
Successive extractions on the nested cylinders $\Om\times(\frac1n,n)$
and a diagonal selection yield a further subsequence
$(\eps_{j_k})_{k\in\N}$ such that
\be{10.ae_subsequence}
\na u_{\eps_{j_k}}\to\na u
\qquad\mbox{a.e. in }\Om\times(0,\infty).
\ee
The convergences (\ref{11.7}), (\ref{11.3}), and
(\ref{10.strong_original}) remain valid along $(\eps_{j_k})_{k\in\N}$.
Relabelling $(\eps_{j_k})_{k\in\N}$ as $(\eps_j)_{j\in\N}$,
(\ref{10.strong_original}) and (\ref{10.ae_subsequence}) establish
(\ref{10.1}).

\medskip
By (\ref{11.7}), (\ref{11.3}), and (\ref{10.1}), along the relabelled
sequence we have
\be{10.flux_ae}
\sqrt{\veps}\na\ueps-\ueps\sqrt{\veps}\na\veps
\to\sqrt{v}\na u-u\sqrt{v}\na v
\qquad\mbox{a.e. in }\Om\times(0,\infty).
\ee
Finally, by Egorov's theorem, the uniform $L^2$ bound (\ref{1.6}) and the almost-everywhere convergence
(\ref{10.flux_ae}) imply (\ref{10.6}).
\qed
\setcounter{equation}{0}
\section{Solution property of the limit functions}\label{s6}
In the solution concept below we will make use of two auxiliary functions. For $\lam>1$ and $s\ge0$ we first let
\be{12.1}
f_\lambda(s):=\frac{1}{(s+1)^\lam}, 
\ee
so that
\be{12.15}
f_\lambda'(s)=-\frac{\lam}{(1+s)^{\lam+1}}<0
\qquad\mbox{and}\qquad
f_\lambda''(s)=\frac{\lam(\lam+1)}{(1+s)^{\lam+2}}>0\qquad\text{for }s\geq0,
\ee
and then additionally introduce for $\lambda>1$ and $s\geq0$ the function
\be{12.16}
\Phi_\lambda(s):=\int_0^s \sigma^2f_{\lambda}''(\sigma)d\sigma.
\ee
Evidently, by choice of $\lambda>1$ we have $f_\lambda, \Phi_\lambda\in C^\infty([0,\infty))\cap L^\infty((0,\infty))$. 
Moreover, these functions satisfy 
\begin{align*}
0\leq\Phi_\lambda(s)\leq \frac{\lambda(\lambda+1)}{\lambda-1}
\qquad\text{and }\qquad
|s^2f_\lambda'(s)|\leq \lambda,\qquad\text{for }s\geq0,
\end{align*}
which will ensure that the corresponding terms in Definition~\ref{dw} are well defined. We use the energy functional $\mathcal{F}$ from (\ref{df}) on the finite-energy class $\mathcal{E}$ introduced there, with the convention $0\ln0=0$.
\begin{defi}\label{dw}
  Let $\Om\subset\R^2$ be a bounded domain with smooth boundary. Assume (\ref{init}), and let
  \be{w1}
	\left\{ 
	\begin{array}{l}
	u\in L_{loc}^\infty([0,\infty);L^1(\Om))
	 \cap L^2_{loc}((0,\infty); W^{1,2}(\Om))
	\qquad \mbox{and} \\[1mm]
	v\in L_{loc}^2([0,\infty);W^{1,2}(\Om))
	\ear
  \ee
  be such that $u\ge 0$ and $v\ge 0$ a.e. in $\Om\times(0,\infty)$, that for every $\lambda>1$, with $f_\lambda$ and $\Phi_\lambda$ as in \eqref{12.1} and \eqref{12.16}, respectively, we have
  \be{w2}
  uf_\lambda'(u)v\na u\in L_{loc}^1(\bom\times[0,\infty);\R^2),
  \qquad
  \qquad
  uf_\lambda''(u)v|\na u|^2\in L_{loc}^1(\bom\times[0,\infty)),
  \ee
  as well as
  \be{w22}
  uv\in L_{loc}^1(\bom\times[0,\infty)),
   \quad
   \sqrt{v}\na u-u\sqrt{v}\na v\in L^2_{loc}(\bom\times[0,\infty);\R^2),
   \quad\mbox{and}\quad
  v_t\in L_{loc}^2(\bom\times[0,\infty)).
  \ee 
  Let $\mathcal{F}$ be as defined in (\ref{df}). Then $(u,v)$ will be called a {\em global generalized energy solution} of (\ref{0}) if there exists a null set $N_\star\subset(0,\infty)$ such that
\[
(u(\cdot,t),v(\cdot,t))\in\mathcal{E}
\qquad\mbox{for all }t\in(0,\infty)\setminus N_\star,
\]
and
 \be{w4}
   \mathcal{F}(u(\cdot,t),v(\cdot,t))
   +\int_{0}^{t}\io v_{t}^2
    + \int_{0}^{t}\io \Big|\sqrt{v}\na u-u\sqrt{v}\na v\Big|^2
   \le \mathcal{F}(u_0,v_0)
 \qquad\mbox{for all }t\in(0,\infty)\setminus N_\star,
   \ee
and
\be{w6}
\io u(\cdot,t)
=\io u_0
\qquad\mbox{for all }t\in(0,\infty)\setminus N_\star,
\ee
if for all $\vp\in C^\infty_0(\bom\times[0,\infty))$ one has
\be{w5}
-\int_0^\infty\io v\vp_t
-\io v_0\vp(\cdot,0)
=
-\int_0^\infty\io \nabla v\cdot\nabla\vp
+\int_0^\infty\io (u-v)\vp,
\ee  
if for any $p>1$ $\lim_{t \searrow 0}\|v(\cdot,t)-v_0\|_{L^p(\Om)}=0,$ 
and if for every $\lambda>1$ and every nonnegative $\vp\in C_0^\infty(\bom\times [0,\infty))$, we have
  \begin{align}\label{w7}
  &\ -\!\int_0^\infty\io f_\lambda(u)\vp_t-\!\io f_\lambda(u_0)\vp(\cdot,0)\nn\\
     \le&\ -\! \int_0^\infty\!\io u f_\lambda''(u)v|\na u|^2\vp
     -\! \int_0^\infty\!\io \Phi_\lambda(u)v v_t\vp
     +  \int_0^\infty\!\io \Phi_\lambda(u)uv\vp
      -\!  \int_0^\infty\!\io \Phi_\lambda(u)v^2\vp\nn\\
     &\ -\! \int_0^\infty\!\io \Phi_\lambda(u)|\na v|^2\vp
     -\! \int_0^\infty\!\io u f_\lambda'(u)v\na u\cdot\na\vp 
       +\int_0^\infty\!\io \Big\{u^2f_\lambda'(u)v
          -\! \Phi_\lambda(u)v\Big\} \na v\cdot\na\vp.
  \end{align}
\end{defi}
The regularity in (\ref{w1}) gives $u(\cdot,t),v(\cdot,t)\in L^2(\Om)$ for almost every $t>0$.
The bound $|s\ln s|\le s^2+e^{-1}$ for $s\ge0$ and the Cauchy--Schwarz inequality therefore imply
$(u(\cdot,t),v(\cdot,t))\in\mathcal{E}$ for almost every $t>0$.
The null set $N_\star$ may be enlarged to include the exceptional time slices.\abs
{\bf Remark.} \quad
By adapting the argument in \cite[Lemma 2.1]{win_sima2015}, one can verify
that the above notion of generalized energy solution is consistent with
classical solvability. More precisely, if a generalized energy solution
$(u,v)$ satisfies
$(u,v)\in (C^0(\bom\times[0,\infty)) \cap C^{2,1}(\bom\times(0,\infty)))^2$,
then $(u,v)$ is a classical solution of (\ref{0}).\abs
Evidently, many of the conditions above, in particular those on the regularity of a
 global generalized energy solution, can be easily inferred for our limit functions
  from Lemma~\ref{lem11}. We shall now focus on verifying the
  remaining requirements in the definition. To this end, we first derive an
  identity for $\int_\Om \partial_t f_\lambda(\ueps)\varphi$.
\begin{lem}\label{lem12}
Let $\lam>1$, $T\in(0,\infty]$, and let $f_\lambda,\Phi_\lambda$ be provided by \eqref{12.1} and \eqref{12.16}, respectively. Then,
if $\vp\in C^1(\bom\times[0,T))$, then 
\bea{12.2}
 \io \pa_t f_\lambda(\ueps)\vp
   &=&-\io \ueps f_\lambda''(\ueps)\veps|\na\ueps|^2\vp
   -\io \Phi_\lambda(\ueps)\veps v_{\eps t}\vp
   + \io \Phi_\lambda(\ueps)\ueps\veps\vp
   - \io \Phi_\lambda(\ueps)\veps^2\vp\nn\\
   & & -\io \Phi_\lambda(\ueps)|\na\veps|^2\vp
   +\io \Big\{\ueps^2f_\lambda'(\ueps)\veps
   - \Phi_\lambda(\ueps)\veps\Big\}
   \na\veps\cdot\na\vp\nn\\
   & &-\io \ueps f_\lambda'(\ueps)\veps\na\ueps\cdot\na\vp
   \qquad \mbox{for all $t\in(0,T)$ and }\eps\in(0,1).
\eea
\end{lem}
\proof
  Using (\ref{0eps}) and integrating by parts gives
   \bea{12.3}
   \io \pa_t f_\lambda(\ueps)\vp
   &=&-\io \{f_\lambda''(\ueps)\vp\na\ueps
   +f_\lambda'(\ueps)\na\vp\}
   \cdot \{\ueps\veps\na\ueps-\ueps^2\veps\na\veps\}\nn\\
   &=& -\io \ueps f_\lambda''(\ueps)\veps|\na\ueps|^2\vp
   + \io \ueps^2f_\lambda''(\ueps)\veps(\na\ueps\cdot\na\veps)\vp\nn\\
   & &-\io \ueps f_\lambda'(\ueps)\veps\na\ueps\cdot\na\vp
   +\io \ueps^2f_\lambda'(\ueps)\veps\na\veps\cdot\na \vp
   \eea
for all $t\in(0,T)$ and $\eps\in(0,1)$. From (\ref{12.1}), we also note that
$\na\Phi_\lambda(\ueps)=\ueps^2f_\lambda''(\ueps)\na\ueps$. Therefore,
using the second equation of (\ref{0eps}), we compute
\begin{align*}
&\io \ueps^2f_\lambda''(\ueps)\veps(\na\ueps\cdot\na\veps)\vp\\
=\ & -\io \Phi_\lambda(\ueps)\na\cdot(\veps\vp\na \veps)\\
=\ & -\io \Phi_\lambda(\ueps)\{\veps\vp\Del\veps+|\na\veps|^2\vp
+\veps\na\veps\cdot\na\vp\}\\
=\ & \io \Phi_\lambda(\ueps)\veps\vp\cdot\{-v_{\eps t}+\ueps-\veps\}
-\io \Phi_\lambda(\ueps)|\na\veps|^2\vp- \io \Phi_\lambda(\ueps)\veps\na\veps\cdot\na\vp\\
=\ &  -\io \Phi_\lambda(\ueps)\veps v_{\eps t}\vp
+ \io \Phi_\lambda(\ueps)\ueps\veps\vp
- \io \Phi_\lambda(\ueps)\veps^2\vp-\io \Phi_\lambda(\ueps)|\na\veps|^2\vp
- \io \Phi_\lambda(\ueps)\veps\na\veps\cdot\na\vp
\end{align*}
for all $t\in(0,T)$ and $\eps\in(0,1)$. Substituting this into (\ref{12.3}) immediately yields
(\ref{12.2}).
\qed
Apparently, Lemmas~\ref{lem11} and \ref{lem10} do not allow us to directly pass to the limit in some of the integrals appearing above, which is why we derive the following final pair of estimates.
\begin{lem}\label{lem13}
Assume $r\in[1,2)$ and $\lam>1$, and let $f_\lambda$ be as in \eqref{12.1}. Then for any $T>0$, there exists $C=C(T,\lambda,r)>0$ such that 
\be{13.0}
\int_0^T\io\ueps f_\lambda''(\ueps)\veps|\na\ueps|^2
+\int_0^T\io|\ueps f_\lambda'(\ueps)\veps\na\ueps|^r
\le C
\qquad\mbox{for all }\eps\in(0,\epsr).
\ee
\end{lem}
\proof
We first recall from (\ref{1.2})--(\ref{1.4}) that there exists $c_1>0$
   such that
   \be{13.2}
   \sup_{t>0}\left\{\io \ueps(\cdot,t)\veps(\cdot,t)
   +\io \veps^2(\cdot,t)
   +\io |\na\veps(\cdot,t)|^2\right\}
   +\int_0^T\io v_{\eps t}^2
   \le c_1
   \ee
   for all $T>0$.
   Applying Lemma \ref{lem12} with $\vp=1$ then yields
   \bea{13.1}
    \frac{d}{dt}\io f_\lambda(\ueps)
      &=&-\io \ueps f_\lambda''(\ueps)\veps|\na\ueps|^2
      -\io \Phi_\lambda(\ueps)\veps v_{\eps t}
      + \io \Phi_\lambda(\ueps)\ueps\veps
      - \io \Phi_\lambda(\ueps)\veps^2\nn\\
      & & -\io \Phi_\lambda(\ueps)|\na\veps|^2
 \qquad\mbox{for all $t\in(0,T)$.}
   \eea
   
By the boundedness of $\Phi_\lambda$, Young's inequality together with (\ref{13.2}) provides $c_2=c_2(\lambda)>0$ such that
\bas
& &\hs{-15mm}
\bigg| -\io \Phi_\lambda(\ueps)\veps v_{\eps t}
      + \io \Phi_\lambda(\ueps)\ueps\veps
      - \io \Phi_\lambda(\ueps)\veps^2
      -\io \Phi_\lambda(\ueps)|\na\veps|^2\bigg|\\
& \le&c_2\cdot\bigg\{
\io \veps^2
+\io v_{\eps t}^2
+\io|\na\veps|^2
+\io \ueps\veps
\bigg\}\\
&\le& c_2\cdot\left(c_1+\io v_{\eps t}^2\right).
\eas
Integrating (\ref{13.1}) over $(0,T)$ and using (\ref{13.2}) therefore gives
\bea{13.6}
\int_0^T\io\ueps f_\lambda''(\ueps)\veps|\na\ueps|^2
&\le& c_1c_2(1+T)
+\io f_\lambda(u_{0\eps})
-\io f_\lambda(\ueps(\cdot,T))\nn\\
&\le& c_1c_2(1+T)
+|\Om|.
\eea
due to $0<f_\lambda(s)\le1$ for all $s\ge0$.  Therefore, an application of Young's inequality entails
\bas
\io|\ueps f_\lambda'(\ueps)\veps\na\ueps|^r
\le \io f_\lambda''(\ueps)\ueps\veps|\na\ueps|^2
+ \io \ueps^\frac{r}{2-r}|f_\lambda'(\ueps)|^\frac{2r}{2-r}\veps^\frac{r}{2-r}
|f_\lambda''(\ueps)|^{-\frac{r}{2-r}}
\eas
for all $t\in(0,T)$. Because of the precise forms of $f_\lambda'$ and $f_\lambda''$ stated in \eqref{12.15} and the fact that for $\lam>1$ and $r\in[1,2)$, 
\be{13.4}
\frac{r}{2-r}-(\lam+1)\cdot\frac{2r}{2-r}+(\lam+2)\cdot\frac{r}{2-r}
=-\frac{r(\lam-1)}{2-r}<0,
\ee
we may choose $c_3=c_3(\lam,r)>0$ such that
\begin{align*}
s^\frac{r}{2-r}|f_\lambda'(s)|^\frac{2r}{2-r}|f_\lambda''(s)|^{-\frac{r}{2-r}}
&=s^\frac{r}{2-r}\cdot(1+s)^{-\frac{2r(\lambda+1)}{2-r}}\cdot(1+s)^{\frac{r(\lambda+2)}{2-r}}\cdot \lambda^{\frac{r}{2-r}}\cdot(1+\lambda)^{-\frac{r}{2-r}}\\
&\le (1+s)^{-\frac{r(\lambda-1)}{2-r}}\lambda^\frac{r}{2-r}\le c_3
\qquad\mbox{for all $s\ge0$.}
\end{align*}
 Consequently,
\bas
\int_0^T\io|\ueps f_\lambda'(\ueps)\veps\na\ueps|^r
\le \int_0^T\io f_\lambda''(\ueps)\ueps\veps|\na\ueps|^2
+ c_3\cdot \int_0^T\io \veps^\frac{r}{2-r}.
\eas
Combining this with \eqref{13.2}, \eqref{13.6} and an application of Hölder's inequality completes the proof of (\ref{13.0}).
\qed
Now everything is prepared to verify \eqref{w7} and the other conditions of Definition~\ref{dw}.%
\begin{lem}\label{lem14}
Let $(u,v)$ be as constructed in Lemma \ref{lem11}. Then $(u,v)$ is a global
generalized energy solution in the sense of Definition \ref{dw}. Moreover, for every $0<\tau<T$ and every $\vp\in C_0^\infty(\bom\times(\tau,T))$,
\be{14.u_weak}
-\int_\tau^T\io u\vp_t
=-\int_\tau^T\io uv\na u\cdot\na\vp
+\int_\tau^T\io u^2v\na v\cdot\na\vp.
\ee
\end{lem}
\proof
Throughout the proof we work along the final subsequence from Lemma~\ref{lem10},
   again denoted by $(\eps_j)_{j\in\N}$, which retains the convergences
   (\ref{11.2})--(\ref{11.6}). By Lemmas \ref{lem1}, \ref{lem11}, and \ref{lem10}, the regularity
      requirements (\ref{w1}) and (\ref{w22}) are satisfied. The mass
      conservation property (\ref{w6}) follows directly from
      (\ref{11.10}) in Lemma~\ref{lem11}.
   Moreover, for each $\lambda>1$ and $f_\lambda,\Phi_\lambda$ accordingly provided by \eqref{12.1} and \eqref{12.16}, combining Lemmas 
   \ref{lem11} and \ref{lem10} with
   Lemma \ref{lem13} and invoking Fatou's lemma yields (\ref{w2}). In particular, \eqref{11.7}, \eqref{11.3} and \eqref{10.1} imply that $$\ueps f_\lambda'(\ueps)\veps\nabla\ueps\to  u f_\lambda'(u) v \na u\quad\text{a.e. in }\Omega\times(0,\infty)\text{ as }\eps=\eps_j\searrow0,$$ which paired with the fact that for some $r\in(1,2)$ Lemma~\ref{lem13} provides $C(T)>0$ such that $$\int_0^T\io|\ueps f_\lambda'(\ueps)\veps\nabla\ueps|^r\leq C(T)$$ and an application of the Vitali convergence theorem shows that,
  as $\eps=\eps_j\searrow 0$,
   \be{14.1}
   \ueps f_\lambda'(\ueps)\veps\na \ueps
   \to u f_\lambda'(u) v \na u
   \qquad\mbox{in } L^1_{loc}(\bom\times[0,\infty)).
   \ee
   
   \medskip
   
   To verify (\ref{w7}),  for any given nonnegative 
   $\vp\in C_0^\infty(\bom\times [0,\infty))$, applying Lemma \ref{lem12} with this
    test function and integrating in time and then performing
    an integration by parts, we obtain
   \bea{14.2}
    \int_0^\infty\io \ueps f_\lambda''(\ueps)\veps|\na\ueps|^2\vp 
    &=&\int_0^\infty\io f_\lambda(\ueps)\vp_t
    +\io f_\lambda(u_{0\eps})\vp(\cdot,0)
    -\int_0^\infty\io \Phi_\lambda(\ueps)\veps v_{\eps t}\vp\nn\\
    & &+ \int_0^\infty\io \Phi_\lambda(\ueps)\ueps\veps\vp
    - \int_0^\infty\io \Phi_\lambda(\ueps)\veps^2\vp
    -\int_0^\infty\io \Phi_\lambda(\ueps)|\na\veps|^2\vp\nn\\
    & &-\int_0^\infty\io \ueps f_\lambda'(\ueps)\veps\na\ueps\cdot\na\vp\nn\\
    & &+\int_0^\infty\io \Big\{\ueps^2f_\lambda'(\ueps)\veps
        - \Phi_\lambda(\ueps)\veps\Big\}
        \na\veps\cdot\na\vp
   \eea
for all $\eps\in(0,\epsr)$. The boundedness and continuity of $f_\lambda(s)$, $\Phi_\lambda(s)$,
and $s^2f_\lambda'(s)$ on $[0,\infty)$ as ensured by
(\ref{12.1})--(\ref{12.16}) and the bounds following (\ref{12.16}), together with Lemma \ref{lem11}, 
imply that, as $\eps=\eps_j\searrow 0$,
\bas
& &\Phi_\lambda(\ueps)
\wsto \Phi_\lambda(u)
\qquad\mbox{in } L^\infty_{loc}(\bom\times[0,\infty))\nn\\
& & \Phi_\lambda(\ueps)\veps
\to \Phi_\lambda(u)v
\qquad\mbox{in } L^2_{loc}(\bom\times[0,\infty))\nn\\
& & \Phi_\lambda(\ueps)\ueps\veps
\to  \Phi_\lambda(u)uv
\qquad\mbox{in } L^1_{loc}(\bom\times[0,\infty))\nn\\
& &  \Phi_\lambda(\ueps)\veps^2
\to  \Phi_\lambda(u)v^2
\qquad\mbox{in } L^1_{loc}(\bom\times[0,\infty))\nn\\ 
& & \Phi_\lambda(\ueps)|\na\veps|^2
\to  \Phi_\lambda(u)|\na v|^2
\qquad\mbox{in } L^1_{loc}(\bom\times[0,\infty))
\qquad\mbox{and}\nn\\ 
& & \Big\{\ueps^2f_\lambda'(\ueps)\veps
- \Phi_\lambda(\ueps)\veps\Big\}\na\veps
\to \Big\{u^2f_\lambda'(u)v
- \Phi_\lambda(u)v\Big\}\na v
\qquad\mbox{in } L^1_{loc}(\bom\times[0,\infty)).
\eas
Accordingly,  since $f_\lambda''>0$ and $\vp\ge0$, combining these convergences with (\ref{14.1}), Fatou's lemma, and
the pointwise convergences recorded in Lemmas \ref{lem11} and
\ref{lem10}, passing to the limit $\eps=\eps_j\searrow 0$ in (\ref{14.2}) yields
\bas
\int_0^\infty\io u f_\lambda''(u)\, v\, |\na u|^2 \vp
 &\le& \liminf_{\eps=\eps_j\searrow 0}
  \int_0^\infty\io \ueps f_\lambda''(\ueps)\veps|\na\ueps|^2\vp \\
      &=&\int_0^\infty\io f_\lambda(u)\vp_t
      +\io f_\lambda(u_{0})\vp(\cdot,0)
      -\int_0^\infty\io \Phi_\lambda(u)v v_t\vp\nn\\
      & &+ \int_0^\infty\io \Phi_\lambda(u)u v\vp
      - \int_0^\infty\io \Phi_\lambda(u)v^2\vp
      -\int_0^\infty\io \Phi_\lambda(u)|\na v|^2\vp\nn\\
      & &-\int_0^\infty\io u f_\lambda'(u)v\na u\cdot\na\vp\nn\\
      & &+\int_0^\infty\io \Big\{u^2f_\lambda'(u)v
          - \Phi_\lambda(u)v\Big\}
          \na v\cdot\na\vp,
\eas
which is precisely (\ref{w7}). Since $\lambda>1$ was arbitrary and no further subsequence was selected, (\ref{w2}) and (\ref{w7}) hold for every $\lambda>1$ for the same pair $(u,v)$.

 \medskip
 
To prove (\ref{w5}), we first test the second equation in (\ref{0eps}) by an
arbitrary $\vp\in C^\infty_0(\bom\times[0,\infty))$ and integrate by parts.
Then, using the definition of $v_{0\eps}$ in (\ref{0eps}), the regularity
assumption on $v_0$ in (\ref{init}), and the convergence properties in
(\ref{11.2}), (\ref{11.8}), and (\ref{11.4}), we may pass to the limit
$\eps=\eps_j\searrow0$ in the resulting weak formulation.  
 \medskip

It remains to verify the energy inequality (\ref{w4}). The boundedness of the initial data in (\ref{init}) ensures that $(u_0,v_0)\in\mathcal{E}$.
Moreover, the positive-time regularity in (\ref{11.1}) gives
$(u(\cdot,t),v(\cdot,t))\in\mathcal{E}$ for every $t>0$.
By Lemmas \ref{lem11} and \ref{lem10},
combined with the regularity of the initial data from (\ref{init}), we have, for each fixed $t>0$, 
as $\eps=\eps_j\searrow 0$,
\bas
\mathcal{F}(\ueps(\cdot,t),\veps(\cdot,t))
&\to& \mathcal{F}(u(\cdot,t),v(\cdot,t)),\\
\mathcal{F}(u_{0\eps},v_{0\eps})
&\to& \mathcal{F}(u_0,v_0).
\eas
For the second convergence, note that $u_{0\eps}=u_0+\eps$ and
$v_{0\eps}=v_0+\eps$ converge uniformly to $u_0$ and $v_0$, respectively,
and $\na v_{0\eps}=\na v_0$.
Since $s\mapsto s\ln s$, extended by zero at $s=0$, is continuous on every
compact interval in $[0,\infty)$,
$\io u_{0\eps}\ln u_{0\eps}\to\io u_0\ln u_0$, while the remaining terms
in $\mathcal{F}(u_{0\eps},v_{0\eps})$ converge directly.
Moreover, the weak convergence statements in (\ref{11.5}) and (\ref{10.6}) yield
\bas
\int_{0}^{t}\io v_{t}^2
+ \int_{0}^{t}\io \Big|\sqrt{v}\na u-u\sqrt{v}\na v\Big|^2
\le
\liminf_{\eps=\eps_j\searrow 0}
\Big\{ \int_{0}^{t}\io v_{\eps t}^2
 + \int_{0}^{t}\io \Big|\sqrt{\veps}\na\ueps-\ueps\sqrt{\veps}\na\veps\Big|^2\Big\}
\eas
by weak lower semicontinuity of the norm. Consequently, we infer from taking the limit in \eqref{3.3} that
\bas
\mathcal{F}(u(\cdot,t),v(\cdot,t))
+\int_{0}^{t}\io v_{t}^2
+ \int_{0}^{t}\io \Big|\sqrt{v}\na u-u\sqrt{v}\na v\Big|^2
\le  \mathcal{F}(u_0,v_0) 
\qquad\mbox{for all } t>0,
\eas
which shows (\ref{w4}). Therefore, the limit pair $(u, v)$ 
constructed in Lemma \ref{lem11} is indeed a global generalized energy solution in
the sense of Definition \ref{dw}.
\medskip 

Finally, fix $0<\tau<T$. The convergences (\ref{11.7}), (\ref{11.3}), and (\ref{10.1}) imply that, as $\eps=\eps_j\searrow0$,
\be{14.u_flux}
\ueps\veps\na\ueps-\ueps^2\veps\na\veps
\to uv\na u-u^2v\na v
\qquad\mbox{in }L^2(\Om\times(\tau,T);\R^2).
\ee
Testing the first equation in (\ref{0eps}) by $\vp\in C_0^\infty(\bom\times(\tau,T))$ and integrating by parts gives
\[
-\int_\tau^T\io\ueps\vp_t
=-\int_\tau^T\io
\big(\ueps\veps\na\ueps-\ueps^2\veps\na\veps\big)\cdot\na\vp.
\]
Passing to the limit using (\ref{11.7}) and (\ref{14.u_flux}) proves (\ref{14.u_weak}).\qed
\setcounter{equation}{0}
\section{The initial trace of $u$}\label{s7}
We now establish that the generalized energy solution attains
the prescribed initial cell density. Although $u_0$ enters the
renormalized formulation through (\ref{w7}), its strong attainment
at the initial time remains to be proved. Combining this inequality
with convexity and mass conservation, we show that
\bas
u(\cdot,t)\to u_0
\quad\mbox{in }L^1(\Om)
\qquad\mbox{as }t\searrow0,
\eas
complementing the initial behavior of $v$ established in
(\ref{11.9}).\abs
We first extract one-sided initial information from (\ref{w7}) using a time cutoff; cf. \cite[proof of Lemma 2.1]{win_sima2015}. Motivated by the convex-entropy argument for initial traces in \cite[Section 6(e)]{diperna1985}, we divide the convexity remainder by $-f_\lambda'(u_0)$. Its linear term then becomes $u(\cdot,t)-u_0$, whose integral vanishes by mass conservation.
\begin{lem}\label{lem15}
Let $(u,v)$ be as constructed in Lemma \ref{lem11}. For $\lam>1$, let $\fl$ be as in (\ref{12.1}) and set
\bas
H_\lambda(s,r):=\frac{\fl(s)-\fl(r)-\fl'(r)(s-r)}{-\fl'(r)}
\qquad\mbox{for }s,r\ge0.
\eas
Then
\be{15.1}
\lim_{t\searrow0}\io H_\lambda(u(\cdot,t),u_0)=0
\qquad\mbox{for every }\lam>1.
\ee
\end{lem}
\proof
   Fix $\lam>1$ and any nonnegative $\psi\in C^\infty(\bom)$. For $\del>0$ and $t_0>0$, let $\xi=\xi_{\del,t_0}\in C_0^\infty([0,\infty))$ satisfy $\xi=1$ on $[0,t_0]$,
   $\xi=0$ on $[t_0+\del,\infty)$, and $\xi'\le0$ on $[0,\infty)$.

   \medskip

   Setting $\vp(x,t):=\psi(x)\xi(t)$ gives a nonnegative test function in $C_0^\infty(\bom\times[0,\infty))$.
   Since $-\xi'\ge0$ is supported in $[t_0,t_0+\del]$ and $\int_0^\infty(-\xi')=1$, the continuity in (\ref{11.1}) yields
   \bea{15.2}
   -\int_0^\infty\io \fl(u)\vp_t
   &=&-\int_0^\infty\xi'(t)\Big\{\io\fl(u(\cdot,t))\psi\Big\}dt\nn\\
   &\to&\io\fl(u(\cdot,t_0))\psi
   \qquad\mbox{as }\del\searrow0
   \eea
   for each fixed $t_0>0$.

   \medskip

   Let
   \bas
   \mathcal R:=|u\fl'(u)v\na u|+|vv_t|+uv+v^2+|\na v|^2+|v\na v|.
   \eas
   Since Lemma \ref{lem14} provides (\ref{w1}), (\ref{w2}), and (\ref{w22})
   and since Young's inequality yields
   \bas
   |vv_t|+|v\na v|\le v^2+\frac12v_t^2+\frac12|\na v|^2,
   \eas
   it follows that $\mathcal R\in L^1(\Om\times(0,T))$ for every $T>0$.
   Using $\varphi(x,t)=\psi(x)\xi(t)$ as a test function in (\ref{w7}), discarding the nonpositive term $-\int_0^\infty\io u\fl''(u)v|\na u|^2$, and recalling the bounds for $\Phi_\lambda$ and $s^2\fl'(s)$ following (\ref{12.16}), we obtain $c_1=c_1(\lam)>0$ such that
   \bas
   -\int_0^\infty\io\fl(u)\vp_t-\io\fl(u_0)\psi
   \le c_1\|\psi\|_{W^{1,\infty}(\Om)}
   \int_0^{t_0+\del}\io\mathcal R.
   \eas
   Accordingly, in view of \eqref{15.2}, letting $\del\searrow0$ gives
   \bas
   \io\Big(\fl(u(\cdot,t_0))-\fl(u_0)\Big)\psi
   \le c_1\|\psi\|_{W^{1,\infty}(\Om)}\int_0^{t_0}\io\mathcal R
   \qquad\mbox{for each fixed }t_0>0.
   \eas
   Hence, by the integrability of $\mathcal R$,
   \be{15.3}
   \limsup_{t\searrow0}\io\Big(\fl(u(\cdot,t))-\fl(u_0)\Big)\psi\le0
   \qquad\mbox{for every nonnegative }\psi\in C^\infty(\bom).
   \ee

   \medskip

   By (\ref{init}) and (\ref{12.15}),
   \bas
   \frac1{-\fl'(u_0)}=\frac{(1+u_0)^{\lam+1}}{\lam}\in C^0(\bom)
   \eas
   is strictly positive. Extending this function continuously to a neighborhood of $\bom$ and mollifying provides nonnegative $\psi_k\in C^\infty(\bom)$ such that
   \bas
   \left\|\psi_k-\frac1{-\fl'(u_0)}\right\|_{L^\infty(\Om)}\to0
   \qquad\mbox{as }k\to\infty.
   \eas
   Since $0<\fl\le1$, (\ref{15.3}) implies,
   \bas
   \limsup_{t\searrow0}\io\frac{\fl(u(\cdot,t))-\fl(u_0)}{-\fl'(u_0)}
   &=& \limsup_{t\searrow0}\bigg\{\io \Big(\fl(u(\cdot,t))-\fl(u_0)\Big)
   \cdot \Big(\frac{1}{-\fl'(u_0)}-\psi_k+\psi_k\Big)
   \bigg\}\\ 
   &\le&  \limsup_{t\searrow0}\io \Big(\fl(u(\cdot,t))-\fl(u_0)\Big)
   \cdot \Big(\frac{1}{-\fl'(u_0)}-\psi_k\Big)\\
   & &+ \limsup_{t\searrow0}\io \Big(\fl(u(\cdot,t))-\fl(u_0)\Big)
      \cdot \psi_k\\
   &\le& 2|\Om|\left\|\psi_k-\frac1{-\fl'(u_0)}\right\|_{L^\infty(\Om)}
   \qquad\mbox{for every $k\in\N$}.
   \eas
   Letting $k\to\infty$ therefore yields
   \be{15.4}
   \limsup_{t\searrow0}\io\frac{\fl(u(\cdot,t))-\fl(u_0)}{-\fl'(u_0)}\le0.
   \ee

   \medskip

   By (\ref{12.15}), $\fl$ is convex and $-\fl'>0$, whence $H_\lambda(s,r)\ge0$ for $s,r\ge0$.
   Using the mass identity (\ref{11.10}), we obtain
   \bas
   0\le\io H_\lambda(u(\cdot,t),u_0)
   &=&\io\frac{\fl(u(\cdot,t))-\fl(u_0)}{-\fl'(u_0)}
      +\io(u(\cdot,t)-u_0)\\
   &=&\io\frac{\fl(u(\cdot,t))-\fl(u_0)}{-\fl'(u_0)}
   \qquad\mbox{for all }t>0.
   \eas
   Combining this with (\ref{15.4}) proves (\ref{15.1}).
\qed
We next combine an explicit lower bound for $H_\lambda$ with the Cauchy--Schwarz inequality and mass conservation to convert the vanishing defect into strong convergence to $u_0$ in $L^1(\Omega)$.
\begin{lem}\label{lem16}
Let $(u,v)$ be as constructed in Lemma \ref{lem11}. Then
\be{16.1}
\lim_{t\searrow0}\|u(\cdot,t)-u_0\|_{L^1(\Om)}=0.
\ee
\end{lem}
\proof
   Fix $\lam>1$. Convexity of $s\mapsto s^\lam$ on $(0,\infty)$ gives
   \bas
   \left(\frac{1+r}{1+s}\right)^\lam-1
   \ge\lam\left(\frac{1+r}{1+s}-1\right)
   \qquad\mbox{for }s,r\ge0.
   \eas
   Thus the definition of $H_\lambda$ in Lemma \ref{lem15} and (\ref{12.1}), (\ref{12.15}) imply
   \bea{16.2}
   H_\lambda(s,r)
   &=&\frac{1+r}{\lam}\left\{\left(\frac{1+r}{1+s}\right)^\lam-1\right\}+s-r\nn\\
   &\ge&(1+r)\left(\frac{1+r}{1+s}-1\right)+s-r
   =\frac{(s-r)^2}{1+s}
   \qquad\mbox{for }s,r\ge0.
   \eea
   By the Cauchy--Schwarz inequality and the mass identity (\ref{11.10}),
   \bas
   \|u(\cdot,t)-u_0\|_{L^1(\Om)}^2
   &=& \bigg\{\io \frac{|u(\cdot,t)-u_0|}{\sqrt{1+u(\cdot,t)}}
   \cdot \sqrt{1+u(\cdot,t)}
   \bigg\}^2\\
   &\le& \bigg\{\io\frac{|u(\cdot,t)-u_0|^2}{1+u(\cdot,t)}\bigg\}
          \bigg\{\io(1+u(\cdot,t))\bigg\}\\
   &\le&\Big\{|\Om|+\io u_0\Big\}\io H_\lambda\big(u(\cdot,t),u_0\big)
   \qquad\mbox{for all }t>0.
   \eas
   Now (\ref{15.1}) yields (\ref{16.1}).
\qed

Gathering the previous results concludes the proof of Theorem~\ref{theo11}.

\textbf{Proof of Theorem \ref{theo11}.} \quad
  By the convergence property provided by (\ref{11.7}) we infer from (\ref{mass})
  and (\ref{1.1}) that (\ref{t2}) and (\ref{t3}) indeed hold. In both cases \textup{(i)} and \textup{(ii)},
  the regularity in (\ref{t1}) and the initial trace (\ref{t4}) follow from Lemmas \ref{lem11} and \ref{lem10},
  while Lemma \ref{lem14} provides the generalized energy solution and the weak formulation (\ref{14.u_weak}).
  The strong initial trace (\ref{t.u_initial}) is established in Lemma \ref{lem16}. Furthermore, the radial symmetry of 
  $(u(\cdot,t), v(\cdot,t))$ about $x=0$ for all $t>0$, stated in \textup{(ii)} of
  Theorem  \ref{theo11}, is again a consequence of Lemma \ref{lem11} together 
  with the radial symmetry of $(\ueps(\cdot,t), \veps(\cdot,t))$ for all $t>0$,
  established in Lemma \ref{lem_loc}.
  \qed

\bigskip

\section*{Acknowledgments}
T.B. acknowledges support from the {\em Deutsche Forschungsgemeinschaft} (Project No.~462888149). G.L. was supported by the National Natural Science Foundation of China (Grant No.~12401262) and the Fundamental Research Funds for the Central Universities (Project No.~B260201046).\smallskip

The authors thank the anonymous referees for their valuable comments and suggestions in preparing the manuscript.

{\small
 }

\end{document}